\documentclass[12pt]{amsart}
\usepackage{amssymb,amsmath,amsfonts,amsthm,bbm,mathrsfs,times,graphicx,color,comment, enumerate}
\usepackage[utf8]{inputenc}
\usepackage{tikz}
\usetikzlibrary{arrows,shapes}
\usetikzlibrary{fadings}
\usetikzlibrary{intersections}
\usetikzlibrary{decorations.pathreplacing}
\usetikzlibrary{quotes}
\usetikzlibrary{angles}
\usetikzlibrary{babel}

\usepackage{extarrows}% http://ctan.org/pkg/extarrows

\usepackage{mathtools,mathrsfs,times,subfigure,graphicx}
\usepackage{epsfig,psfrag,epsf}
\usepackage{xcolor, array}

\usepackage[super]{nth}

\usetikzlibrary{shapes.misc}
\usetikzlibrary{decorations.markings}
\tikzset{cross/.style={cross out, draw=black, minimum size=2*(#1-\pgflinewidth), inner sep=0pt, outer sep=0pt},
	cross/.default={1pt}}

\usepackage{calc}% http://ctan.org/pkg/calc
\usepackage{ifthen}

\usepackage{extarrows}% http://ctan.org/pkg/extarrows

\newcommand\numberthis{\addtocounter{equation}{1}\tag{\theequation}}

\usepackage{mathtools}
\DeclareMathOperator{\re}{\mathbb{R}e}
\DeclareMathOperator{\im}{\mathbb{I}m}
\newcommand{\lnorm}[1]{ \left\| #1 \right\|}

\newcommand{\spr}[2]{\langle #1,#2 \rangle}

\renewcommand{\i}{\mathrm i}

\newcommand{\INF}{{\infty}}

\newcommand{\supp}{\mbox{supp}}

\newcommand{\tta}{\theta}

\newcommand{\OM}{\Omega}

\newcommand{\sph}{{{\mathbb S}^ 1}}

\newcommand{\del}{\partial}

\newcommand{\Gam}{\varGamma}
\newcommand{\ol}{\overline}
\newcommand{\ds}{\displaystyle}
\newcommand{\dba}{\overline{\partial}}

\newcommand{\BR}{\mathbb{R}}
\newcommand{\BC}{\mathbb{C}}

\newcommand{\BZ}{\mathbb{Z}}
\newcommand{\BN}{\mathbb{N}}

\newcommand{\bu}{{\bf u}}

\newcommand{\bg}{{\bf g}}

\newcommand{\bbf}{{\bf f}}
\newcommand{\bh}{{\bf h}}

\newcommand{\bF}{{\bf F}}

\newcommand{\bzero}{\mathbf 0}

\newcommand{\btheta}{\boldsymbol \theta}

\newcommand{\jpn}{\langle n\rangle}
\newcommand{\jpk}{\langle k\rangle}
\newcommand{\jpj}{\langle j\rangle}

\newcommand{\LL}{{\mathcal L}}
\newcommand{\B}{\mathcal{B}}

\newcommand{\HT}{\mathcal{H}}

\newtheorem{theorem}{Theorem}[section]
\newtheorem{prop}{Proposition}[section]
\newtheorem{lemma}{Lemma}[section]
\newtheorem{cor}{Corollary}[section]
\newtheorem{claim}{Claim}[section]

\newtheorem{remark}{Remark}[section]

\numberwithin{equation}{section}

\title[On the range of the $X$-ray transform]{On the range of the $X$-ray transform of symmetric tensors compactly supported in the plane}
\begin{document}
	\date{\today}
	\author{Kamran Sadiq}
	%\address{Johann Radon Institute of Computational and Applied Mathematics (RICAM), Altenbergerstrasse 69, 4040 Linz, Austria}
	%\email{kamran.sadiq@ricam.oeaw.ac.at}
	\address{Faculty of Mathematics, University of Vienna, Oskar-Morgenstern-Platz 1, 1090 Vienna, Austria}
	\email{kamran.sadiq@univie.ac.at}
	\author{Alexandru Tamasan}
	\address{Department of Mathematics, University of Central Florida, Orlando, 32816 Florida, USA}
	\email{tamasan@math.ucf.edu}
	%    General info
	%%% Adding these to change 1991 to 2020 Math Subject Classification
	\makeatletter
	\@namedef{subjclassname@2020}{\textup{2020} Mathematics Subject Classification}
	\makeatother
	\maketitle
	
	%%%%%%%%%%%%%%%  ABSTRACT %%%%%%%%%%%%%%%%%%%%%%%%%%
	\begin{abstract}
		We find the necessary and sufficient conditions on the Fourier coefficients of a function $g$ on the torus to be in the range of the $X$-ray transform of a symmetric tensor of compact support in the plane. 
		%The smallest order can be easily read from the data.
	\end{abstract}
	
	%%%%%%%%%%%%%%%% INTRODUCTION %%%%%%%%%%%%%%%%%%%%%%%
	\section{Introduction}
	
	We revisit the range characterization of the $X$-ray transform of a real valued symmetric $m$- tensors of compact support in the Euclidean plane. Via a re-parametrization of lines, the $m=0$ case is the classical Radon transform \cite{radon1917}, for which the necessary and sufficient constraints have been long established independently by Gelfand and Graev \cite{gelfandGraev}, Helgason  \cite{helgason65}, and Ludwig \cite{ludwig}; we refer to the result as  \textit{GGHL characterization}.  Models which account for the attenuation have also been considered in the homogeneous case \cite{kuchmentLvin}, and in the non-homogeneous case in the breakthrough works \cite{ABK, novikov01,novikov02}, and subsequently \cite{natterer01, bomanStromberg, bal04, monard17}. Finding applications in noise reduction \cite{yuWang07, yuWangEtall06}, completion of the data \cite{xia_etall, gompelDelfriseDyck, karp_etal, kudoSaito},  CT-hardware failure diagnosis \cite{patch}, the range characterization problem in the $0$-tensor case continues to stimulate the interests of both mathematicians and practitioners \cite{chanLeng05,clackdoyle13, kazantsevBukhgeimJr04, kazantsevBukhgeimJr06, sadiqTamasan01, monard16, donsubPreprint}. In particular in \cite{sadiqTamasan01} the authors gave a range characterization in terms of the Bukhgeim-Hilbert transform, the Hilbert-like transform associated with A-analytic maps in the sense of Bukhgeim \cite{bukhgeimBook}. The latter result  was extended to $1$-and $2$-tensors in \cite{sadiqTamasan02, sadiqScherzerTamasan}, and to an arbitrary order in  \cite{omogbheSadiq}.
	
	For tensors of order $m\geq 1$, the non-injectivity of the $X$-ray transform makes  the range characterization problem more interesting.  In the Euclidean plane, the GGHL-characterization was extended to arbitrary symmetric $m$-tensors in \cite{pantyukhina}.
	For a survey of results on tomography of tensors in the Euclidean plane we refer to \cite{derevtsovSvetov15}.
	
	The systematic study of tensor tomography in non-Euclidean spaces originated in \cite{vladimirBook}. On simple Riemannian surfaces, the range characterization of the geodesic $X$-ray of compactly supported $0$ and $1$ has been established in terms of the scattering relation in \cite{pestovUhlmann}, and the results were extended to symmetric tensors of arbitrary order in \cite{AMU}, and to attenuating media  in \cite{venke20}; see \cite{paternainSaloUhlmann14} for a comprehensive survey. 
	
	The connection between the Euclidean version of the characterization in \cite{pestovUhlmann} and the GGHL characterization was established in \cite{monard16}.
	
	In \cite{sadiqTamasan22} the authors considered the lines parametrized by points on the torus and gave a range characterization for compactly supported functions in terms of the Fourier coefficients on the Fourier lattice of the torus. This novel point of view allowed to establish the missing connection between the result in \cite{sadiqTamasan01} and the classical GGHL characterization. Although $Xf$  is a function on the torus, this problem differs from the one in \cite{ilmavirta_etal}, where for a given direction (of rational slope) the integration takes place over a finite union of parallel segments in the unit disc.

	In here we extend the result in \cite{sadiqTamasan22} to symmetric tensors of an arbitrary order. The method of proof is  based on some explicit mapping properties of the Bukhgeim-Hilbert transform. Apart from the symmetry constraints due to the double parameterization of the lines, of specific interest are the moment  conditions. 
	The thrust of this work are the constraints \eqref{RTCond_Eventensor} and \eqref{RTCond_Oddtensor} replacing the generalized moment conditions in \cite{pantyukhina},  and the sufficiency part in Theorems \ref{RangeCharac_Eventensor} and \ref{RangeCharac_Oddtensor} for tensors of finite smoothness.

	%Section 2 -- \ref{sec:prelim}
	%Section 3 - \ref{sec:L2map}
	%Section 4 - \ref{sec:newBHprop}
	%Section 5 - \ref{Sec:pf_mainTh}
	
	This work concerns real valued tensors. For a complex valued tensor $\bbf$, since $\ds \re (X\bbf) = X(\re(\bbf))$ and $\ds \im (X\bbf) = X(\im(\bbf))$,  Theorems \ref{RangeCharac_Eventensor} and \ref{RangeCharac_Oddtensor} %and Corollaries \ref{corollary} \ref{corollary_Oddtensor}
	 apply separately to  $\re (\bbf)$ and to $\im(\bbf)$.

	All the details establishing notation and the statement of the main results are in Section \ref{sec:prelim}. In Section \ref{sec:L2map} we briefly recall existing results on $A$-analytic maps that are used in the proofs. 
%	In Section \ref{sec:newBHprop} we establish a new mapping property of the Bukhgeim-Hilbert transform, which is key to the proof of our main result in Section \ref{Sec:pf_mainTh}. 
	In Section \ref{Sec:pf_mainTh} we present the proof of the Theorems \ref{RangeCharac_Eventensor} and \ref{RangeCharac_Oddtensor}. 
	To improve the readability of the work, some of the claims are proven in the appendix.

	\section{Preliminaries and statement of main results}\label{sec:prelim}
	
	%We identify the plane $\BR^2$ by the complex plane $\BC$, $z^1 \equiv z = x^1+\i x^2, z^2 \equiv \bar{z} = x^1-\i x^2$.
	%the unit disc $$\OM = \{ z  \in \BC : |z| <1  \}.$$ 
	
	%Let us consider the Cartesian coordinate system $(x^1,x^2)$ on the plane $\BR^2$.
	Let $\bbf= (f_{i_1i_2...i_m})$, with $i_1,...,i_m\in\{1,2\}$ be a real valued symmetric $m$-tensor, with integrable components of compact support in $\BR^2$. By scaling and translating, we may assume that all the components have compact support inside the the unit disc $\OM = \{ z  \in \BC : |z| <1  \}.$ The boundary $\Gam$ of $\OM$ is the unit circle, but we keep this notation to differentiate from the set $\sph$ of directions. The symmetry refers to the components $f_{i_1i_2...i_m}$ being invariant under any transposition of indexes.
	
	For $\ds \btheta^m := \underbrace{ \btheta  \otimes \btheta \otimes \cdots \otimes \btheta}_m\in\left(\sph\right)^m$ and $z\in\OM$, let 
	$\langle \bbf(z), \btheta^m \rangle$ denote the pairing
	\begin{align}
		\langle \bbf(z), \btheta^m \rangle = f_{i_1 \cdots i_m} (z) \theta^{i_1} \cdot \theta^{i_2} \cdots \theta^{i_m},%\quad z\in\OM,
	\end{align} where the summation convention is understood over all repeated indexes $\ds(i_1,i_2, \cdots, i_m)\in \{1,2\}^m$.

	The $X$-ray transform of $\bbf$ is given by
	\begin{equation}\label{Xf} 
		\begin{aligned}
			X\bbf(z,\btheta) &:=\int_{-\infty}^{\infty} \langle \bbf(z+t\btheta), \btheta^m \rangle dt,\quad (z,\btheta)\in\OM\times\sph.
		\end{aligned}
	\end{equation}
	Following directly from its definition
	\begin{align}\label{eq:xrayTransform_hom}
		X\bbf(x, -\btheta)= (-1)^m X\bbf(x,\btheta),
		%, \quad (\lambda\btheta \in \sph).
	\end{align}so that $\btheta \mapsto X\bbf(x,\btheta)$ is an even function for  $\bbf$ of an even order $m$, and an odd function for $\bbf$ of an odd order $m$.

	Lines $\ds L_{(\beta,\theta)}:=\{ e^{i\beta}+s e^{i\theta}: \; s\in\BR\}$ intersecting $\overline\OM$ are parametrized in coordinates $\ds \{(e^{i\beta},e^{i\theta}): \beta,\theta\in(-\pi,\pi]\}$ on the torus  $\Gam\times\sph$, and
	then   
	\begin{align}\label{eq:xrayTransform}
		X \bbf (e^{i \beta},e^{i \theta})  
		=\int_{-\infty}^{\infty} \langle \bbf ( e^{i \beta} +se^{i \theta}), \btheta^m \rangle ds
	\end{align}is also understood as a function on the torus.
	%  $\Gam\times\sph$, and where $\langle\cdot,\cdot \rangle$ is the inner product in \eqref{innerprod}. 
	
	%From the above \eqref{eq:xrayTransform}, the function $ X\bbf$ has the following homogeneity in its second argument:
	%\begin{align}\label{eq:xrayTransform_hom}
	%X \bbf (e^{i \beta},t e^{i \theta})= \frac{t^m}{|t|}X \bbf (e^{i \beta}, e^{i \theta}), \quad (0 \neq t \in \BR).
	%\end{align}
	
	Since $\displaystyle L_{(\beta,\theta)}=L_{(2\theta -\beta-\pi,\theta)}= L_{(\beta,\theta+\pi)}=L_{(2\theta -\beta-\pi,\theta+\pi)},$ the set of lines intersecting $\overline\OM$ are quadruply covered when $(e^{i\beta},e^{i\theta})$ ranges over the entire torus $ \Gam\times \sph$. Moreover, the following symmetries are satisfied,
	\begin{align}
		&X \bbf (e^{i \beta},e^{i \theta})= (-1)^m X \bbf (e^{i \beta},e^{i (\theta+\pi)}), \text{ and } \label{X_ray_sym**}\\
		&X \bbf (e^{\i \beta},e^{ \i \theta})=(-1)^{m} X \bbf (e^{i(2\theta -\beta-\pi)},e^{i( \theta+\pi)}), \text{ for }  (e^{i \beta},e^{i \theta})\in\Gam\times\sph 
		\label{Xray_symm*};
	\end{align}
	see Figure \ref{fig:fanbeam1} below.

	%Note that $X \bbf $ is even in $\tta$ for symmetric $m$-tensor field $\bbf$ of even ($m$ even) order, and $X \bbf $ is odd in $\tta$ for symmetric $m$-tensor field $\bbf$ of  odd ($m$ odd) order.
	
	%%%%%%%%%%%%%%%%%%%%%%%% Figure in the Fan-beam Geometry
	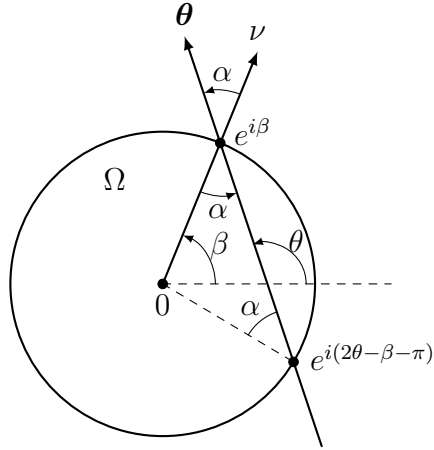
\begin{figure}[ht]
		\centering
		
		% parameters 
		\pgfmathsetmacro{\Radius}{1.35} 
		\pgfmathsetmacro{\RRadius}{1.5*\Radius}
		\pgfmathsetmacro{\RRadiustheta}{0.67*\Radius}
		\pgfmathsetmacro{\Height}{1.25}
		\pgfmathsetmacro{\startAngle}{asin(\Height/\Radius)}
		
		%   \pgfmathsetmacro{\lensCenter}{cos(\startAngle)*\lensRadius}
		\pgfmathsetmacro{\ptAngle}{\startAngle}
		\pgfmathsetmacro{\alphaAngle}{0.6*\startAngle}
		\pgfmathsetmacro{\thetaAngle}{\startAngle + \alphaAngle} %% \theta = \beta +\alpha
		\pgfmathsetmacro{\gammaAngle}{\startAngle + 2*\alphaAngle-180} %% \gamma = 2\alpha + \beta -pi 
		\pgfmathsetmacro{\omegaAngle}{\thetaAngle - 90} %% \omega = \theta - \pi/2
		\pgfmathsetmacro{\xbpt}{cos(\ptAngle)*\Radius}
		\pgfmathsetmacro{\ybpt}{sin(\ptAngle)*\Radius} 
		
		\pgfmathsetmacro{\xpt}{cos(\ptAngle)*1.1*\RRadius}
		\pgfmathsetmacro{\ypt}{sin(\ptAngle)*1.1*\RRadius} 
		
		\pgfmathsetmacro{\Lxpt}{\xbpt + cos(\thetaAngle)}
		\pgfmathsetmacro{\Lypt}{\ybpt + sin(\thetaAngle)}
		
		\pgfmathsetmacro{\Lxxpt}{\xbpt - cos(\thetaAngle)*2.1*\Radius}
		\pgfmathsetmacro{\Lyypt}{\ybpt - sin(\thetaAngle)*2.1*\Radius}
		
		\pgfmathsetmacro{\omegaxpt}{cos(\omegaAngle)*\RRadius}
		\pgfmathsetmacro{\omegaypt}{sin(\omegaAngle)*\RRadius}
		
		\pgfmathsetmacro{\gammaxpt}{cos(\gammaAngle)*\Radius}
		\pgfmathsetmacro{\gammaypt}{sin(\gammaAngle)*\Radius}

		\begin{tikzpicture}[scale=1.5,cap=round,>=latex]
			
			\tikzset{
				thick/.style=      {line width=0.8pt},
				very thick/.style= {line width=1.1pt},
				ultra thick/.style={line width=1.6pt}
			}
			
			%%% Defining the origin and drawing a circle os radius \Radius
			\coordinate[label=below:$0$] (O) at (0,0,0);
			\filldraw[black] (O) circle(1.2pt);
			\draw[thick] (O) circle (\Radius cm);
			
			%%% Defining the point and labeling Omega
			\coordinate[label=above:$\Omega$] (OM) at(120:0.85cm);
			
			%%% The path from center to the horixontal axis
			%%%%% Defining the point to label the angle \beta and horizontal axis
			\coordinate (E) at (\RRadius,0,0);
			\draw[dashed] (O) -- (E);
			
			%%%%% Defining the point to label the angle \theta and horizontal axis
			\coordinate (E1) at (\Radius,0,0);
			\coordinate (E2) at (\RRadiustheta,0,0);
			
			%%%%% Defining the point on the boundary of the circle and labelling it
			\coordinate (P) at (\xbpt,\ybpt);
			\filldraw[black] (P) circle(1.2pt);
			
			%%%%% Labelling the point z=e^{i \beta} 
			\coordinate[label=right:$e^{ i \beta}$] (Z) at (\xbpt+0.025,\ybpt+0.1);
			
			%%%%% Defining the other point on the boundary of the circle and labelling it
			\coordinate (P1) at (\gammaxpt,\gammaypt);
			\filldraw[black] (P1) circle(1.2pt);
			\draw[dashed] (O) -- (P1);
			
			%%%%% Defining the normal \nu at P and drawing the line from origin to \nu
			\coordinate[label=above:$\nu$] (nu) at (\xpt,\ypt);
			\draw[thick,->] (O) -- (nu);
			
			%%%%% Labelling the point on the boundary  z1 = e^{i (2\theta -\beta-\pi)}
			\coordinate[label=right:$e^{i (2\theta -\beta-\pi)}$] (Z1) at (\gammaxpt+0.05,\gammaypt+0.05);

			%%%%% Defining the line L passing through P and with direction \alpha
			\coordinate[label=above:$\btheta$] (L) at  (\Lxpt,\Lypt);
			%\coordinate[label=above:$l_1$] (l) at  (\Lxpt,\Lypt);
			%\coordinate[label=above:$\nu$] (nu) at (\xpt,\ypt);
			\draw[thick,->] (P) -- (L);
			
			%%%%% Defining the line L passing through P and with direction \alpha
			\coordinate (LL) at  (\Lxxpt,\Lyypt);
			%\coordinate[label=above:$\nu$] (nu) at (\xpt,\ypt);
			\draw[thick,-] (P) -- (LL);
			
			%%%%% Defining the line L passing through P and with direction \alpha
			%\coordinate[label=below:$\vec{\omega}$] (W) at  (\omegaxpt,\omegaypt);
			%\coordinate[label=above:$\nu$] (nu) at (\xpt,\ypt);
			%\draw[thick,->] (O) -- (W);
			
			%%%%%%% For the braces %%%%%%%% 
			\tikzset{
				position label/.style={
					above = 3pt,
					text height = 1.5ex,
					text depth = 1ex
				},
				brace/.style={
					decoration={brace,mirror},
					decorate
				}
			}  
			%%%%%%%%%%% Drawing the braces and labeling the angles 
			%   \draw [brace,decoration={raise=0.5ex}] [brace]  (P.north) -- (dipole.north) node [position label, pos=0.4, rotate = 25,scale=0.8] {$\fobj$};
			%   \draw [brace,decoration={raise=0.5ex}] [brace]  (O.south) -- (P.south) node [position label, below=0.1, pos=0.6, rotate = 45,scale=0.8] {$r=\abs{\br}$};
			\pic[draw,->,angle radius=.7cm,angle eccentricity=1.3,"$\beta$"] {angle=E--O--P}; 
			\pic[draw,->,angle radius=.55cm,angle eccentricity=1.29,"$\theta$"] {angle=E1--E2--P}; 
			\pic[draw,->,angle radius=.7cm,angle eccentricity=1.3,"$\alpha$"] {angle=nu--P--L}; 
			\pic[draw,->,angle radius=.7cm,angle eccentricity=1.3,"$\alpha$"] {angle=O--P--LL};
			\pic[draw,-,angle radius=.7cm,angle eccentricity=1.3,"$\alpha$"] {angle=P--P1--O};
			%   \pic[draw,->,angle radius=.7cm,angle eccentricity=1.3,"$2\alpha+\beta-\pi$"] {angle=P1--O--E}; 
			
		\end{tikzpicture}
		\caption{Fan-beam coordinates: $e^{i \beta}\in\Gam$,  $e^{i\theta}\in\sph$, and $\btheta = (\cos \theta, \sin \theta)$.
		} \label{fig:fanbeam1}
	\end{figure}

	If $\bbf= (f_{i_1i_2...i_m})$, with $i_1,...,i_m\in\{1,2\}$ is merely integrable in $\OM$, then $X \bbf$ may not be integrable on the torus. However, if  either \begin{align}\label{eq:f_refularity_cond1}
		\textnormal{supp } f_{i_1...i_m}  \subset \OM, %\{ z: \lvert z \rvert \leq \delta  \}  \mbox{ for some } 0 < \delta <1, 
		\mbox{ or }  f_{i_1...i_m}   \in L^p(\OM)\mbox{ for some }p >2,
		\end{align}
		then $X\bbf \in L^1(\Gamma\times\sph)$; see Proposition \ref{prop:fLp_regularity} in the appendix.
	
%		\begin{prop}\label{prop:fLp_regularity}
%		Let $\bbf \in L^1(S^m;\OM)$ be a symmetric $m$-tensor with integrable components. Assume that each of its components satisfies one of the conditions:
%		\begin{align}\label{eq:f_refularity_cond}
%		\textnormal{supp }f_{i_1...i_m}  \subset \{ z: \lvert z \rvert \leq\sqrt{ 1- \delta^2}  \}, \;  0 < \delta <1, \quad \textnormal{or} \quad  f_{i_1...i_m}\in L^p(\OM), \;p >2,
%		\end{align}
%		then $X\bbf  \in L^1(\Gamma\times\sph)$.
%	\end{prop}

	We consider the partition of the torus into three parts: the ``outflux" part
	\begin{equation}\label{eq:Gam+}
		\Gam_+:=
		\left \{(e^{i \beta} ,e^{i(\beta+\alpha )} )\in \Gam \times\sph:\,  \beta\in (-\pi, \pi] , \; |\alpha|<\frac{\pi}{2} \right \},
	\end{equation}the  ``influx" part
	\begin{equation}\label{eq:Gam-}
		\Gam_-:=
		\left \{(e^{i \beta} ,e^{i(\beta+\alpha )} )\in \Gam \times\sph:\,  \beta\in (-\pi, \pi] , \; \frac{\pi}{2} < |\alpha|\leq\pi  \right \},
	\end{equation}and the (Lebesgue negligible) variety $\Gam_0:=(\Gam\times\sph)\setminus(\Gam_+\cup\Gam_-)$ parameterizing the tangent lines to the circle; see Figure \ref{fig:fanbeam1}.

	Our result gives necessary and sufficient conditions for a function $g\in L^1(\Gam\times\sph)$ to satisfy
	\begin{equation}\label{eq:g_mEven}
		g=
		\left\{
		\begin{array}{ll}
			X \bbf ,&\mbox{ on } \Gam_+,\\
			-X \bbf, &\mbox{ on } \Gam_-.
		\end{array}
		\right.
	\end{equation}
	The characterization is in terms of the Fourier coefficients
	\begin{align}\label{eq:gnk}
		g_{n,k} := \frac{1}{(2 \pi)^2} \int_{-\pi}^{\pi} \int_{-\pi}^{\pi} g(e^{i \beta}, e^{i\theta}) e^{-i n \theta}  e^{-i k \beta} d \theta d \beta,\; n,k \in \BZ
	\end{align}on the lattice $\BZ \times \BZ$. 
	
	The two indexes play a different role. Throughout, the first index is the Fourier mode in the angular variable on $\sph$, and we call it an \emph{angular mode}. The second index is the mode in the boundary variable on $\Gam$, and we call it a \emph{boundary mode}.

	The change of parity in the order of the tensor propagates to the statements of the results. For the sake of clarity we separate the two cases. 
	
	\subsection{ The case of an even order $m$-tensor }
	If $m$ is even, it is easy to check that $g$ in \eqref{eq:g_mEven} satisfy the symmetry relation
	\begin{align}\label{sym*}
		g (e^{i \beta},e^{i \theta})= g (e^{i(2\theta -\beta-\pi)},e^{i( \theta+\pi)}),\text{ for a.e. } (e^{i \beta},e^{i \theta})\in\Gam\times\sph.
	\end{align}Motivated by this relation, let $L^1_{\text{sym}}(\OM\times\sph)$ denote the space of integrable functions $g$ on the torus satisfying the symmetry in \eqref{sym*}. Since 
	$\left(e^{i\beta},e^{i \theta}\right)$ and $\left(e^{i(2\theta-\beta-\pi)},e^{( \theta+\pi)}\right)$ are either both in $\Gam_+$, or both in $\Gam_-$, we can consider  the spaces  $L^1_{\text{sym}}(\Gam_\pm)$ of integrable functions on the half-tori  satisfying \eqref{sym*}. Clearly,  $\ds g\in L^1_{sym}(\Gam\times\sph)$ if and only if its restrictions $\ds g|_{\Gam_\pm}\in L^1_{\text{sym}}(\Gam_\pm)$.
	
	%The symmetry relation \eqref{Xray_symm*} yields $X \bbf\in L^1_{\text{sym}}(\Gam\times\sph)$. 
	
	Moreover, since $g$ in \eqref{eq:g_mEven} is odd with respect to the angular variable: 
	\begin{align}\label{sym**}
		g (e^{i \beta},e^{i \theta})=  -g(e^{i \beta},e^{i (\theta+\pi)}),
	\end{align}
	let us consider the subspace $L^1_{\text{sym,odd}}(\Gam\times\sph)$  of functions in $L^1_{\text{sym}}(\Gam\times\sph)$, which, in addition to satisfying \eqref{sym*}, they also satisfy \eqref{sym**}.
	
	%Note the sign difference with \eqref{X_ray_sym**}.
	
	%Note also that for $g \in L^1_{\text{sym,odd}}(\Gam\times\sph)$, $X \bbf \not\in L^1_{\text{sym,odd}}(\Gam\times\sph)$,  since the symmetry relation  \eqref{X_ray_sym**} is broken.  
	
	%In the statements below we use the notations $C^\mu(\OM)$, $0<\mu<1$, for the space of locally H\"older continuous functions,  and  $\jpn=(1+|n|^2)^{1/2}$.
	
	In the statements below we use the notations in \cite{vladimirBook}:
	\begin{align*}
		L^1(\mathbf{S}^m; \OM) =\left \{\bbf =(f_{i_1 \cdots i_m}) \in \mathbf{S}^m( \OM): f_{i_1 \cdots i_m}\in L^1(\OM)\right \}
	\end{align*}
	for the space of real valued, symmetric tensor fields of order $m$ with integrable components.
	Similarly, 
	$C^\mu(\mathbf{S}^m; \OM)$, $0<\mu<1$, denotes the tensor fields of order $m$ with locally H\"older continuous components, and we use the notation $\jpn=(1+|n|^2)^{1/2}$.
	
	%For the sufficiency part of the Theorem \ref{RangeCharac_Eventensor} below, the tensor field of even order $m = 2q$  is not unique if $q\geq 1$. 
	%This non-uniqueness is also parameterized.
	
	%To exihibit this nonuniqueness, for some $g_{-1},g_{-3}, \cdots, g_{-(2q-1)}\in L^1(\Gam)$ to be specified in the Theorem, we define the class
%	However, it is determined by an element in the non-uniqueness class defined 
	% Other than having the traces $u_{-(2j-1)} \big \lvert_{\Gam} = g_{-(2j-1)}, \; 1 \leq j \leq q,\, q \geq 1$, on the boundary, the $q$ many Fourier modes $u_{-(2j-1)}, \; 1 \leq j \leq q, \, q \geq 1$, are unconstrained. They are chosen arbitrarily from the class 
	% $\Psi_{g}^{\text{even}}$ of functions with prescribed trace on the boundary $\Gam$ defined as

%	\begin{align} \label{NART_mEvenPsiClass}
%		\Psi_{g}^{\text{even}}:=
%		&\left \{\left( \psi_{-1}, \psi_{-3}, \cdots \psi_{-(2q-1)}\right) \in 
%			%\left(C^{1, \mu}(\ol\OM; \BC)\right)^{q}, \mu >\frac{1}{2}:
%		\left(W^{1,1}(\OM; \BC)\right)^{q}:
%		%\right. 		&\left. \quad \vphantom{\int} 
%		\psi_{-(2j-1)} \big \lvert_{\Gam}= g_{-(2j-1)} \; 1 \leq j \leq q \right\}.
%	\end{align}
	
%By Sobolev imbedding, $\psi_{-1}, \psi_{-3}, \cdots, \psi_{-(2q-1)}\in C^{1-\frac{2}{p}}(\ol\OM)$, in particular their traces are in $C^{1-\frac{2}{p}}(\Gam)$.   
%	Note that in the 0-tensor case ($q = 0$), there is no such class.
%	In the statements below we use the notation $\jpn=(1+|n|^2)^{1/2}$.
	%%%%%%%%%%%%%%%%%%%%%%%%%%%%%%%
	%%%%%%%%%%%%%%%%%%%%%%%%%%%%%%%%
	\begin{theorem}[Range characterization for even order tensors]\label{RangeCharac_Eventensor}
		
		(i)  %Let $m = 2q, \, q\geq 0$ be an even integer.
		Let $\bbf \in 	L^1(\mathbf{S}^m; \OM) $ be a real valued, integrable symmetric tensor field of even order $m = 2q, \, q\geq 0$, satisfying \eqref{eq:f_refularity_cond1},  and $g\in L^1_{\emph{sym,odd}}(\Gam\times\sph)$, with $$g= X \bbf \text{ on }\Gam_+ (\mbox{ and  }g=- X \bbf\mbox{ on }\Gam_-).$$
		Then  the Fourier coefficients $\{g_{n,k}\}_{n, k \in \BZ}$ of $g$ satisfy the following conditions:
		\begin{alignat}{3}
			&\text{Oddness}: && \quad g_{n,k}=0, && \quad \text{for all \emph{even} } n\in \BZ \text{, and all } k\in\BZ; \label{RTCond_odd_Eventensor}\\
			&\text{Conjugacy}: && \quad g_{-n,-k}= \ol{g_{n,k}},  &&\quad \text{for all} \;  n, k \in \BZ;\label{RTCond_reality_mEven}\\
			&\text{Symmetry}:
			&& \quad g_{n,k}= (-1)^{n+k}g_{n+2k,-k}, &&\quad \text{for all} \; n,k\in\BZ; \label{RT_FourierOdd_Eventensor}\\
			&\text{Moments}: &&\quad  g_{n,k}=  (-1)^k  g_{n+2k,-k}, &&\quad  \text{for all \emph{odd} } n\leq -(2q+1) \text{,  and all } k\leq 0.\label{RTCond_Eventensor}
		\end{alignat}
		(ii) Let $\{g_{n,k}\}$ be given for all \emph{odd} $n \leq -1$, and $k \in\BZ$ such that 
		\begin{align}\label{gnk_decay}
			\sum_{\substack{n\leq -1\\
					n =  \,\emph{odd}}} \jpn^{2} \sum_{k=-\INF }^\INF \lvert g_{n,k} \rvert < \INF, \quad \text{and} \quad
			% \sum_{n=0, \text{odd}}^{\INF}  \jpn^2 \sum_{k=-\INF}^{\INF} \lvert g_{-n,k} \rvert < \INF, \quad \text{and} \quad
			\sum_{k=-\INF}^\INF  \jpk^{1+\mu}
			\sum_{\substack{n\leq -1\\
					n  =\,\emph{odd}}} \lvert g_{n,k} \rvert < \INF, %\quad 1/2 <\mu <1,
		\end{align}for some $\mu>1/2$, and let 
		\begin{align} \label{g_sequence_mEVEN}
		g_{-n}: = \sum_{k=-\infty}^\infty g_{-n,k} \, e^{i k \beta}, \; \emph{ odd } n \leq -1, 
		\end{align}be defined on $\Gam$. 
		
		If $\{g_{n,k}\}$ satisfy \eqref{RT_FourierOdd_Eventensor} and \eqref{RTCond_Eventensor}, then there exists a real valued 
		$\bbf \in L^1(\mathbf{S}^m; \OM)$
		 such that the mapping %$\{g_{-n,k}\}_{n, k \in \BZ}$ with
		\begin{align}\label{eq:gnk_map}
			(\Gamma\times\sph) \ni (e^{i \beta}, e^{i\theta}) \longmapsto  %g(e^{i\beta},e^{i\theta}):= 
			2\re\left\{\sum_{\substack{n\leq-1\\
					n  = \,\emph{odd}}}  \sum_{k \in\BZ} g_{n,k} e^{i n \theta}  e^{i k \beta}\right\}
			%			+\sum_{\substack{n\geq 1\\
			%			n  = \,\text{odd}}}  \sum_{k \in \BZ} \overline{g_{n,k}} e^{i n \theta}  e^{i k \beta},
		\end{align} 
		defines a function in $L^1_{\emph{sym,odd}}(\Gam\times\sph)$, which coincides with $X \bbf$ on $\Gam_+$ (and with $- X \bbf$ on $\Gam_-$). 
		
		For $q\geq 1$,  $\bbf$  is uniquely determined by an element in the class
		\begin{align} \label{NART_mEvenPsiClass}
		\Psi_{g}^{\text{even}}:=
		&\left \{\left( \psi_{-1}, \psi_{-3}, \cdots \psi_{-(2q-1)}\right) \in 
		%	\left(C^{1, \mu}(\ol\OM; \BC)\right)^{q}, \mu >\frac{1}{2}:
		\left(W^{1,1}(\OM; \BC)\right)^{q}:
		%\right.
	%	\left. \quad \vphantom{\int} 
		\psi_{-(2j-1)} \big \lvert_{\Gam}= g_{-(2j-1)}, \; 1 \leq j \leq q \right\}.
	\end{align} Moreover if $\psi_{-1},\psi_{-3},...,\psi_{-2q+1}\in C^{1,\mu}(\ol\OM)$, then $\bbf\in C^\mu(\mathbf{S}^m; \OM)$. 
	
	If $q=0$, the class is empty and $\bbf$  is uniquely determined by the data. 
	\end{theorem}
	
	The oddness and conjugacy constraints in \eqref{RTCond_odd_Eventensor} and \eqref{RTCond_reality_mEven} are not intrinsic to the $X$-ray transform. The symmetry constraints \eqref{RT_FourierOdd_Eventensor} merely account for each line being doubly parametrized in $\Gam_+$, and they are shared by any function on the torus satisfying the symmetry \eqref{sym*}; see Lemma \ref {lem:evenness} in the appendix.

	The following result is a direct consequence of the algebraic  interaction of the range conditions in \eqref{RTCond_odd_Eventensor}, \eqref{RTCond_reality_mEven}, 
	\eqref{RT_FourierOdd_Eventensor}, and \eqref{RTCond_Eventensor}. To illustrate the result of this interaction,  let us consider the partition of $\ds \BZ^{-}\times\BZ$ as in  Figure \ref{lattice_Eventensor}: 
	
	For an even integer $m \geq 0$,  the white region $W = W^+ \cup W^-$, where
	\begin{equation}\label{White_region}
		\begin{aligned}
			W^+ & := \left \{ (n,k) \in \BZ^{-} \times \BZ^+  \; : \text{odd } n\leq -m-1, \; \text{and} \; 0 \leq k \leq - \frac{n+m+1}{2}            \right \},   \\
			W^- & := \left \{ (n,k) \in \BZ^{-} \times \BZ^-  \; : \text{odd }n\leq -m-1, \; \text{and} \;  k \leq 0            \right \},
		\end{aligned}
	\end{equation}
	and the green region $G=G_L \cup G_R$, where
	\begin{equation}\label{Green_region}
		\begin{aligned}
			G_L & := \left \{ (n,k) \in \BZ^{-} \times \BZ^+  \; : \text{odd }n\leq -1, \; \text{and} \;     \frac{-n+1}{2}  \leq k \leq -n          \right \}.\\
			G_R & := \left \{(n,k) \in \BZ^{-} \times \BZ^+  \; : \text{odd }n\leq -1, \; \text{and} \;   k > -n            \right \}.
		\end{aligned}
	\end{equation}
	Moreover, for even integer $m \geq 2$, the  red  region $R = R^+ \cup R^-$, where 
	\begin{equation}\label{Red_region}
		\begin{aligned}
			R^+ & := \left \{ (n,k) \in \BZ^{-} \times \BZ^+  \; : \text{odd }n\leq -1, \; \text{and} \;  - \frac{n+m-1}{2}  \leq k \leq - \frac{n+1}{2}            \right \},   \\
			R^- & := \left \{ (n,k) \in \BZ^{-} \times \BZ^-  \; : n {\text{ odd}}, -m+1 \leq n\leq -1, \; \text{and} \;  k \leq 0            \right \}.
		\end{aligned}
	\end{equation}
	If $m = 0$, then the corresponding red region  $R$ defined by \eqref{Red_region} is empty.
	
	\begin{remark}\label{InteractionEven}
	The modes in the region $R$ are affected solely by the symmetry \eqref{RT_FourierOdd_Eventensor} due to the double parametrization of the lines in $\Gam_+$. The modes in the region $G$  are affected both by the symmetry and the reality of the tensor  \eqref{RTCond_reality_mEven}. The white region contains the modes affected by the nature of the operator (integration) along the line \eqref{RTCond_Eventensor} in combination with the symmetry. See  Figure \ref{lattice_Eventensor}.
	% and \ref{lattice_Oddtensor},
\end{remark}	
	%\commentK{Have to do the Part (ii)- The above (ii) is for symmetric $0$-tensor field. }
	%%%%%%%%%%%%%%%%%%%%%%%%%%%
	% \begin{comment}
	
	\begin{figure}[ht!]
		\centering
		\begin{tikzpicture}[scale=0.9,cap=round,>=latex]
			% Draw the grid
			\tikzset{help lines/.style={color=blue!50}}
			\draw[thick,step=1cm,help lines] (0,0) grid (15,10);
			\draw[ultra thin,step=1cm,help lines] (0,0) grid (15,10);
			
			% Draw axes
			\draw[dashed, -latex] (15.2,4) -- (-0.7,4) node[anchor=east] {$-n$};
			%		\draw[ultra thick,-latex] (9.2,3) -- (-0.7,3) node[anchor=east] {$-n$};
			\draw[dashed,-latex] (15,0) -- (15,11) node[anchor=east] {$k$};
			% the co-ordinates -- major
			%		\foreach \x in {0,1,...,9} {     % for x-axis
			%			\draw [thick] (\x,0) -- (\x,-0.2);
			%		}
			%		\foreach \y in {0,1,...,9} {   %% for y-axis
			%			\draw [thick] (9,\y) -- (9.1,\y);
			%		}
			% the numbers
			%\foreach \x in {0,1,...,9} { \node [anchor=north] at (\x,-0.3) {\x}; }
			\foreach \x in {6,7}
			{
				\pgfmathtruncatemacro{\rx}{-15+2*\x};
				\node [anchor=north] at (2*\x,3.95)  {\rx};
				\draw [ultra thick] (2*\x,4.1) -- (2*\x,3.9);
			} 
			
			\foreach \x in {0,1,...,5}
			{
				\pgfmathtruncatemacro{\rx}{-15+2*\x};
				%	   	\node [anchor=north] at (2*\x,3.95)  {\rx};
				\draw [ultra thick] (2*\x,4.1) -- (2*\x,3.9);
				%		\ifthenelse{\NOT 0 = \x }{\node [anchor=north] at (2*\x,2.95)  {\rx};}{} % (*)
				%		\ifthenelse{\NOT 0 = \x }{\draw [ultra thick] (2*\x,3.1) -- (2*\x,2.9);}{}
			} 
			
			%% the y numbers  
			\foreach \y in {0,1,...,10}
			{
				\pgfmathtruncatemacro{\rr}{ -4 + \y}; 
				\node [anchor=east] at (15.8,\y)  {\rr};
			} 
			
			% Draw the lines
			% The diagonal black line
			\draw[thick,black] (15,4) to (9,10);
			
			% First red line closest from origin
			\draw[ultra thick,red] (14,4) to (2,10);
			\draw[ultra thick,red] (14,0) to (14,4);
			
			% Second red line closest from origin
			\draw[ultra thick,red] (12,4) to (0,10);
			\draw[ultra thick,red] (12,0) to (12,4);
			\node [anchor=north] at (10,3.95)  {$\cdots \cdots$};
			
			% Last red line closest from origin
			\draw[ultra thick,red] (8,4) to (0,8);
			\draw[ultra thick,red] (8,0) to (8,4);
			\node [anchor=north] at (8,3.95)  {$-m+1$};
			
			%Green line 
			\draw[ultra thick,black] (6,4) to (0,7);
			\draw[ultra thick,black] (6,0) to (6,4);
			\node [anchor=north] at (6,3.95)  {$-m-1$};
			%% Other nodes
			\node [anchor=north] at (4,3.95)  {$-m-3$};
			\node [anchor=north] at (2,3.95)  {$\cdots \cdots $};
			%		\node [anchor=north] at (0,3.95)  {$-m-7$};
			
			% Fill the region
			%		\draw[fill=blue!20!black!30!green,opacity=0.4]  (0,0) -- (6,0) -- (6,4) -- (0,7) -- cycle;
			\draw[fill=red,opacity=0.4]  (6,0) -- (14,0) -- (14,4) -- (2,10) -- (0,10) -- (0,7) -- (6,4)-- cycle;
			\draw[fill=blue!20!black!30!green,,opacity=0.4]  (14,4) -- (14,10) -- (2,10)  -- cycle;
			
			\coordinate[label=below:$n+2k \leq -(m+1)$] (origin) at (2.2,2.8);
			\coordinate[label=below:$n \leq -(m+1)$] (origin) at (2.7,2.2);

			\draw[black] (14,5) circle (5pt);  % g_{-1,1} is at (1,4) node
			\draw[black] (12,7) circle (5pt);  % g_{-3,3} is at (3,6) node
			\draw[black] (10,9) circle (5pt);  % g_{-5,5} is at (5,8) node
			
			\node at (14,6)[draw,rectangle,minimum height=0.1cm]{}; % g_{-1,2} is at (1,3) node
			\node at (12,6)[draw, rectangle,minimum height=0.1cm]{}; % g_{-3,2} is at (3,5) node

			\node[isosceles triangle, isosceles triangle apex angle=60,
			draw, rotate=270,  % 270 degrees rotation
			%fill=violet!50,
			minimum size =0.01cm] (T1)at (14,7){};
			
			\node[isosceles triangle, isosceles triangle apex angle=60,
			draw, rotate=270,  % 270 degrees rotation
			%fill=violet!50,
			minimum size =0.01cm] (T2)at (10,7){};
			
			%%% RED FILLED FIRST
			\node[isosceles triangle, isosceles triangle apex angle=60,
			draw, rotate=240, fill=red!50,
			minimum size =0.01cm] (T3)at (14,3){};
			
			\node[isosceles triangle, isosceles triangle apex angle=60,
			draw, rotate=240,  fill=red!50,
			minimum size =0.01cm] (T4)at (12,5){};
			
			%%% RED FILLED SECOND
			\node[isosceles triangle, isosceles triangle apex angle=60,
			draw, rotate=270, fill=red!50,
			minimum size =0.01cm] (TR2a)at (12,0){};
			
			\node[isosceles triangle, isosceles triangle apex angle=60,
			draw, rotate=270,  fill=red!50,
			minimum size =0.01cm] (TR2b) at (4,8){};
			
			%%% RED FILLED LAST
			\node[isosceles triangle, isosceles triangle apex angle=60,
			draw, rotate=90,  fill=red!50,
			minimum size =0.01cm] (Trlast)at (8,3){};
			
			\node[isosceles triangle, isosceles triangle apex angle=60,
			draw, rotate=90,  fill=red!50,
			minimum size =0.01cm] (Trlast1)at (6,5){};
			
			%%%%%%%%%%%%
			\node at (14,8)[draw,trapezium,minimum height=0.1cm]{}; % g_{-3,1} is at (3,5) node
			\node at (8,8)[draw,trapezium,minimum height=0.1cm]{}; % g_{-3,1} is at (3,5) node
			
			%%% RED FILLED FIRST
			\node at (4,9)[draw,fill=red!50,trapezium,minimum height=0.1cm]{}; 		
			
			%%% RED FILLED SECOND
			\node at (12,1)[draw,fill=red!50,trapezium,minimum height=0.1cm]{}; 
			\node at (6,7)[draw,fill=red!50,trapezium,minimum height=0.1cm]{}; 		
			
			%%% Ellipse
			\draw[black] (12,8) ellipse (3pt and 6pt);
			\draw[black] (10,8) ellipse (3pt and 6pt);
			
			%%% RED FILLED FIRST
			\draw[black,fill=red!50] (14,0) ellipse (6pt and 8pt);
			\draw[black,fill=red!50,] (6,8) ellipse (6pt and 8pt);
			
			%%% RED FILLED SECOND		
			\draw[black,fill=red!50] (12,3) ellipse (9pt and 7pt);
			\draw[black,fill=red!50] (10,5) ellipse (9pt and 7pt);
			\draw[black,fill=red!50]  (2,9) ellipse (3pt and 6pt);
			
			%%% RED FILLED LAST
			\draw[black,fill=red!50] (8,2) ellipse (8pt and 6pt);
			\draw[black,fill=red!50,] (4,6) ellipse (8pt and 6pt);
			
			%%%%%%%%%%%%%%%%%
			
			\node [rectangle, draw, xslant=0.4] at (14,9) {};
			\node [rectangle, draw, xslant=0.4] at (6,9) {};
			
			%%% RED FILLED FIRST
			\node [rectangle, draw, fill=red!50, xslant=-0.3] at (14,1) {};
			\node [rectangle, draw, fill=red!50, xslant=-0.3] at (8,7) {};
			
			%%% RED FILLED SECOND
			\node [rectangle, draw, fill=red!50, xslant=0.4] at (12,2) {};
			\node [rectangle, draw, fill=red!50, xslant=0.4] at (8,6) {};
			
			%%% RED FILLED LAST
			\node [rectangle, draw, fill=red!50, xslant=0.6] at (0,8) {};
			\node [rectangle, draw, fill=red!50, xslant=0.6] at (8,0) {};

			\node[regular polygon, draw,
			regular polygon sides = 5] (p) at (12,9) {};
			
			\node[regular polygon, draw,
			regular polygon sides = 5] (p) at (8,9) {};
			
			%%% RED FILLED LAST
			\node[regular polygon, 	draw, fill=red!50,
			regular polygon sides = 5] (p) at (8,1) {};
			
			\node[regular polygon, draw, fill=red!50,
			regular polygon sides = 5] (p) at (2,7) {};

			%%% RED FILLED FIRST
			\node[regular polygon, 	draw, fill=red!50,
			regular polygon sides = 6] (p) at (14,2) {};
			
			\node[regular polygon, draw, fill=red!50,
			regular polygon sides = 6] (p) at (10,6) {};

			%		\foreach \x in {.5,1.5,...,8.5} {
			%			%\draw [thin] (\x,0) -- (\x,-0.1);
			%		}
			%		\foreach \y in {.5,1.5,...,8.5} {
			%			%\draw [thin] (0,\y) -- (-0.1,\y);
			%		}
			%\caption{l}
		\end{tikzpicture}
		\caption{An even order $m$-tensor field $\bbf$ is determined by the odd negative angular modes on or above the diagonal $k=-n$ (green region), and  the odd negative angular modes (marked red) on the $\frac{m}{2}$ red lines $n+2k=-(m+1)$ for $k\geq0$. All the odd non-positive angular modes on and below the line $n+2k=-(m+1)$, and left of the line 
			$n =-(m+1)$   vanish. %All the odd non-positive angular modes in the white region vanish.	
			\label{lattice_Eventensor}
		}
	\end{figure}
	
	%%%%%%%%%%%%%%%%%%%%%%%%%%%%%%%
	
	%\vspace{1cm}
	%The thrust of this work are the constraints \eqref{RTCond_Eventensor} replacing the moment conditions 
	%(see Remark \ref {GGHL=RTCond}), and the sufficiency part in Theorem \ref{RangeCharac_Eventensor} for functions of finite smoothness. In particular, for $n$ odd, the right hand sides of \eqref{RTCond_Eventensor} and \eqref{RT_FourierEven} differ by a sign. As a direct  consequence, the following holds.

	%For even integer $m \geq 0$, green  region $G = G^+ \cup G^-,$,  yellow  region $Y^+$, and red  region $R$ are express as
	%\begin{align}
	%	G = G^+ \cup G^-,  \quad Y^+ , \quad  R = R^+ \cup R^-,
	%\end{align}  
	%where for even integer $m \geq 0$,  regions $ G^{\pm}$ and $Y^+$ are given by
	
	%%%%%%
	
	\begin{cor}\label{corollary}
		(i) Let $\bbf \in 	L^1(\mathbf{S}^m; \OM) $ be a real valued, integrable symmetric tensor field of even order $m = 2q, \, q\geq 0$. 
		Let $g\in L^1_{\emph{sym,odd}}(\Gam\times\sph)$ coincide with $X \bbf$ on $\Gam_+$, and $\{g_{n,k}\}$ be its Fourier coefficients.  Then, for all \emph{odd} $n\leq -1$,
		\begin{equation}\label{RTCond2negT}
			g_{n,k}=\left\{
			\begin{array}{ll}
				0,& \text{ if }
				(n,k) \in W,\\
				(-1)^{1+k}\overline{g_{-n-2k,k}},& \text{ if }
				(n,k) \in G,\\
				(-1)^{1+k}g_{n+2k,-k},& \text{ if }
				(n,k) \in R,\\
			\end{array}
			\right.
		\end{equation}see Figure \ref{lattice_Eventensor}.

		(ii) Let $\{g_{n,k}\}$ be given for $	(n,k) \in R^+\cup G_L$, such that 
		\begin{align}\label{gnk_decay_Eventensor}
			\sum_{\substack{	(n,k) \in R^+ \cup G_L}} \jpn^{2} \lvert g_{n,k} \rvert < \INF, \quad \text{and} \quad
			\sum_{\substack{	(n,k) \in R^+ \cup G_L}}  \jpk^{1+\mu} 	\lvert g_{n,k} \rvert < \INF, %\quad 1/2 <\mu,
		\end{align}for some $\mu>1/2$. Extend $g_{n,k}$'s from $R^+\cup G_L$ to $R\cup G$ via the relations \eqref{RTCond2negT}. Then there exists a real valued $\bbf \in L^1(\mathbf{S}^m; \OM)$, such that the mapping %$\{g_{-n,k}\}_{n, k \in \BZ}$ with
		\begin{align}\label{eq:gnk_map}
			(\Gamma\times\sph) \ni (e^{i \beta}, e^{i\theta}) \longmapsto  2\re\left\{  \sum_{\substack{	(n,k) \in R \cup G}} g_{n,k} e^{i n \theta}  e^{i k \beta}\right\}
			%			+\sum_{\substack{n\geq 1\\
			%			n  = \,\text{odd}}}  \sum_{k \in \BZ} \overline{g_{n,k}} e^{i n \theta}  e^{i k \beta},
		\end{align} 
		%defines a function $g^*(e^{i \beta}, e^{i\theta}) $, which satisfies \eqref{Xray_defn}.
		is precisely $X \bbf$ on $\Gam_+$ and $-X \bbf$ on $\Gam_-$. 
		For $q\geq 1$,  $\bbf$  is uniquely determined by an element in the class $\Psi_{g}^{\text{even}}$ in \eqref{NART_mEvenPsiClass}. If $q=0$, the zero order tensor is uniquely determined by the data.
\end{cor}

	We formulate next the odd-order tensor case.

	\subsection{ The case of odd order $m$-tensors}
	
	If $m$ is odd, it is easy to check that $g$ in \eqref{eq:g_mEven} obeys the skew-symmetry relation
	\begin{align}\label{skew*}
		g (e^{i \beta},e^{i \theta})= -g (e^{i(2\theta -\beta-\pi)},e^{i( \theta+\pi)}),\text{ for a.e. } (e^{i \beta},e^{i \theta})\in\Gam\times\sph.
	\end{align}
	This motivates to work in the space $L^1_{\text{skew}}(\OM\times\sph)$ of integrable functions $g$ on the torus satisfying the skew-symmetry in \eqref{skew*}.  Since $\left(e^{i\beta},e^{i \theta}\right)$ and $\left(e^{i(2\theta-\beta-\pi)},e^{( \theta+\pi)}\right)$ are either both in $\Gam_+$, or both in $\Gam_-$, we can consider  the spaces  $L^1_{\text{skew}}(\Gam_\pm)$ of integrable functions on the half-tori $\Gam_\pm$ satisfying \eqref{skew*}. Clearly,  $\ds g\in L^1_{skew}(\Gam\times\sph)$ if and only if its restrictions $\ds g|_{\Gam_\pm}\in L^1_{\text{skew}}(\Gam_\pm)$.

	Moreover, since $g$ in \eqref{eq:g_mEven} is even with respect to the angular variable: 
	\begin{align}\label{skew**}
		g (e^{i \beta},e^{i \theta})=  g(e^{i \beta},e^{i (\theta+\pi)}),
	\end{align}
	we further consider the subspace $L^1_{\text{skew,even}}(\Gam\times\sph)$  of functions in $L^1_{\text{skew}}(\Gam\times\sph)$, which, in addition to satisfying \eqref{skew*},  satisfy \eqref{skew**}.

	Our result gives necessary and sufficient conditions for a function $g\in L^1_{\text{skew,even}}(\Gam\times\sph)$ to coincide with $X\bbf$ on $\Gam_+$ (and implicitly with $-X\bbf$ on $\Gam_-$), for some symmetric tensor $\bbf$ of odd order $m$.

	%%%%%%%%%%%%%%%%%%%%%%%%%%%%%%%
	
	\begin{theorem}[Range characterization for odd order tensors]\label{RangeCharac_Oddtensor}
		
		(i)  Let $\bbf \in 	L^1(\mathbf{S}^m; \OM) $ be a real valued, integrable symmetric tensor field of odd order $m = 2q+1, \, q\geq 0$, satisfying \eqref{eq:f_refularity_cond1}, and $g\in L^1_{\emph{skew,even}}(\Gam\times\sph)$, with $$g= X \bbf \text{ on }\Gam_+.$$
		Then  the Fourier coefficients $\{g_{n,k}\}_{n, k \in \BZ}$ of $g$ satisfy the following conditions:
		\begin{alignat}{3}
			&\text{Evenness}: && \quad g_{n,k}=0, && \quad \text{for all \emph{odd} } n\in \BZ \text{, and all } k\in\BZ; \label{RTCond_even_Oddtensor}\\
			&\text{Conjugacy}: && \quad g_{-n,-k}= \ol{g_{n,k}},  &&\quad \text{for all} \;  n, k \in \BZ;\label{RTCond_reality_mOdd}\\
			&\text{Symmetry}:
			&& \quad g_{n,k}= -(-1)^{n+k}g_{n+2k,-k}, &&\quad \text{for all} \; n,k\in\BZ; \label{RT_FourierEven_Oddtensor}\\
			&\text{Moments}: &&\quad  g_{n,k}=  (-1)^k  g_{n+2k,-k}, &&\quad  \text{for all \emph{even} } n\leq -(2q+2) \text{,  and all } k\leq 0.
			\label{RTCond_Oddtensor}
		\end{alignat}
		(ii) Let $\{g_{n,k}\}$ be given for all \emph{even} $n \leq -2$, and $k \in\BZ$ such that 
		\begin{align}\label{gnk_decay_Oddtensor}
			\sum_{\substack{n\leq -2\\
					n =  \,\emph{even}}} \jpn^{2} \sum_{k=-\INF }^\INF \lvert g_{n,k} \rvert < \INF, \quad \text{and} \quad
			\sum_{k=-\INF}^\INF  \jpk^{1+\mu}
			\sum_{\substack{n\leq -2\\
					n  =\,\emph{even}}} \lvert g_{n,k} \rvert < \INF, %\quad 1/2 <\mu <1,
		\end{align}for some $\mu>1/2$,
		and let 
		\begin{align} \label{g_sequence_mODD}
		g_{-n}: = \sum_{k=-\infty}^\infty g_{-n,k} \, e^{i k \beta}, \; \emph{ even } n \leq -2, 
		\end{align}be defined on $\Gam$.

		If $\{g_{n,k}\}$ satisfy \eqref{RT_FourierEven_Oddtensor} and \eqref{RTCond_Oddtensor}, then there exists a real valued $\bbf \in L^1(\mathbf{S}^m; \OM)$ such that the mapping %$\{g_{-n,k}\}_{n, k \in \BZ}$ with
		\begin{align}\label{eq:gnk_map_Oddtensor}
			(\Gamma\times\sph) \ni (e^{i \beta}, e^{i\theta}) \longmapsto  %g(e^{i\beta},e^{i\theta}):= 
			2\re\left\{\sum_{\substack{n\leq-2\\
					n  = \,\emph{even}}}  \sum_{k \in\BZ} g_{n,k} e^{i n \theta}  e^{i k \beta}\right\}
		\end{align} 
		defines a function in $L^1_{\emph{skew,even}}(\Gam\times\sph)$, which coincides with $X \bbf$ on $\Gam_+$ (and with $- X \bbf$ on $\Gam_-$). 
		For $q\geq 0$,  $\bbf$  is uniquely determined by an element in the class
		\begin{align}\label{NART_mOddPsiClass}
	\Psi_{g}^{\emph{odd}}:=
		&\left \{ \left(\psi_0,  \psi_{-2},  \cdots ,\psi_{-2q}\right) \in W^{1,1}(\OM ;\BR)\times \left(W^{1,1}(\OM; \BC)\right)^{q}:  \psi_{-2j} \big \lvert_{\Gam}= g_{-2j}, \; 0 \leq j \leq q \right\}.
	\end{align}  Moreover if $\psi_{0},\psi_{-2},...,\psi_{-2q}\in C^{1,\mu}(\ol\OM)$, then $\bbf\in C^\mu(\mathbf{S}^m; \OM)$. 
\end{theorem}
		
%%%%%%%%%%%%%%%%%%%%%%%%%%%
% ODD TENSOR LATTICE

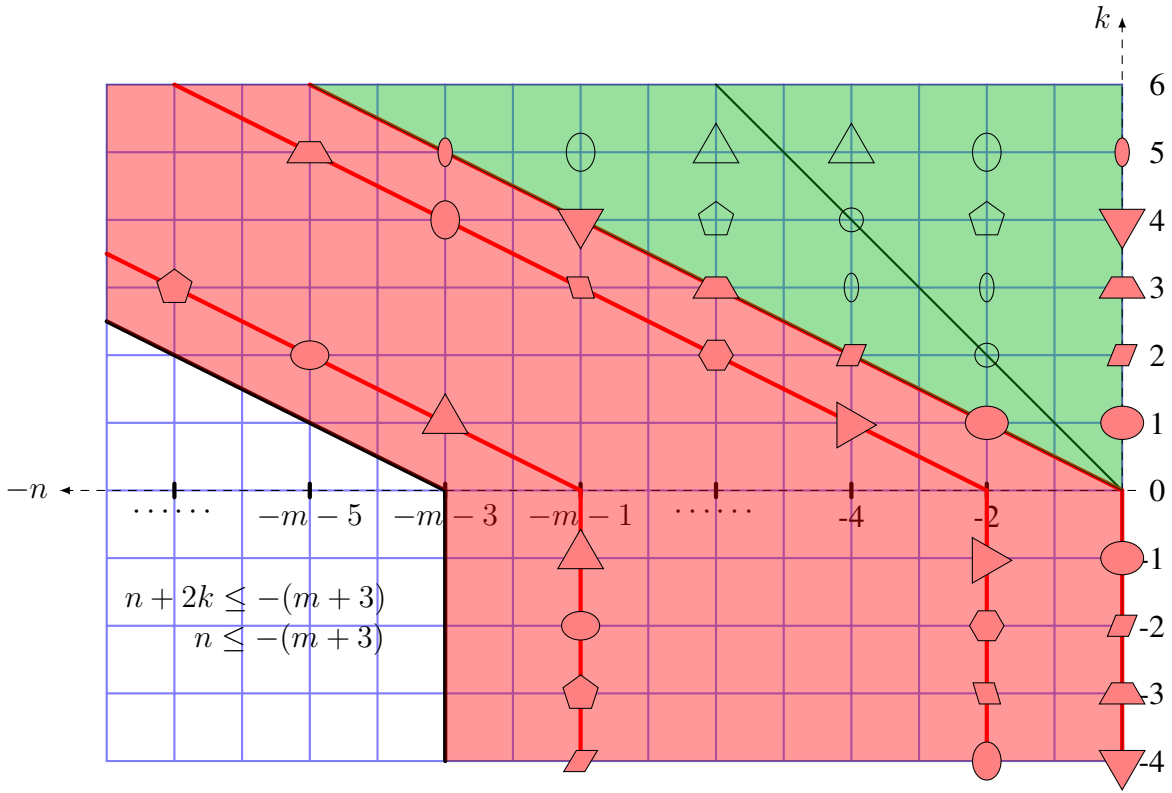
\begin{figure}[ht!]\label{lattice_Oddtensor}
	\centering
	\begin{tikzpicture}[scale=0.9,cap=round,>=latex]
		% Draw the grid
		\tikzset{help lines/.style={color=blue!50}}
		\draw[thick,step=1cm,help lines] (0,0) grid (15,10);
		\draw[ultra thin,step=1cm,help lines] (0,0) grid (15,10);
		
		% Draw axes
		\draw[dashed, -latex] (15.2,4) -- (-0.7,4) node[anchor=east] {$-n$};
		%		\draw[ultra thick,-latex] (9.2,3) -- (-0.7,3) node[anchor=east] {$-n$};
		\draw[dashed,-latex] (15,0) -- (15,11) node[anchor=east] {$k$};
		% the co-ordinates -- major
		%		\foreach \x in {0,1,...,9} {     % for x-axis
		%			\draw [thick] (\x,0) -- (\x,-0.2);
		%		}
		%		\foreach \y in {0,1,...,9} {   %% for y-axis
		%			\draw [thick] (9,\y) -- (9.1,\y);
		%		}
		% the numbers
		%\foreach \x in {0,1,...,9} { \node [anchor=north] at (\x,-0.3) {\x}; }
		\foreach \x in {5,6}
		{
			\pgfmathtruncatemacro{\rx}{-14+2*\x};
			\node [anchor=north] at (2*\x+1,3.95)  {\rx};
			\draw [ultra thick] (2*\x+1,4.1) -- (2*\x+1,3.9);
		}

		\foreach \x in {0,1,...,4}
		{
			\pgfmathtruncatemacro{\rx}{-14+2*\x};
			%	   	\node [anchor=north] at (2*\x,3.95)  {\rx};
			\draw [ultra thick] (2*\x+1,4.1) -- (2*\x+1,3.9);
			%		\ifthenelse{\NOT 0 = \x }{\node [anchor=north] at (2*\x,2.95)  {\rx};}{} % (*)
			%		\ifthenelse{\NOT 0 = \x }{\draw [ultra thick] (2*\x,3.1) -- (2*\x,2.9);}{}
		} 
		
		%% the y numbers  
		\foreach \y in {0,1,...,10}
		{
			\pgfmathtruncatemacro{\rr}{ -4 + \y}; 
			\node [anchor=east] at (15.8,\y)  {\rr};
		} 
		
		% Draw the lines
		% The diagonal black line
		\draw[thick,black] (15,4) to (9,10);
		
		% First red line closest from origin
		\draw[ultra thick,red] (15,4) to (3,10);
		\draw[ultra thick,red] (15,0) to (15,4);
		
		% Second red line closest from origin
		\draw[ultra thick,red] (13,4) to (1,10);
		\draw[ultra thick,red] (13,0) to (13,4);
		
		% Second red line closest from origin
		%		\draw[ultra thick,red] (11,4) to (0,9.5);
		%		\draw[ultra thick,red] (11,0) to (11,4);
		\node [anchor=north] at (9,3.95)  {$\cdots \cdots$};
		
		% Last red line closest from origin
		\draw[ultra thick,red] (7,4) to (0,7.5);
		\draw[ultra thick,red] (7,0) to (7,4);
		\node [anchor=north] at (7,3.95)  {$-m-1$};
		
		%Black line 
		\draw[ultra thick,black] (5,4) to (0,6.5);
		\draw[ultra thick,black] (5,0) to (5,4);
		\node [anchor=north] at (5,3.95)  {$-m-3$};
		%% Other nodes
		\node [anchor=north] at (3,3.95)  {$-m-5$};
		\node [anchor=north] at (1,3.95)  {$\cdots \cdots $};
		%		\node [anchor=north] at (0,3.95)  {$-m-7$};

		% Fill the region
		\draw[fill=red,opacity=0.4]  (5,0) -- (15,0) -- (15,4) -- (3,10) -- (0,10) -- (0,6.5) -- (5,4)-- cycle;
		\draw[fill=blue!20!black!30!green,opacity=0.4]  (15,4) -- (15,10) -- (3,10)  -- cycle;
		
		% Label the region
		\coordinate[label=below:$n+2k \leq -(m+3)$] (origin) at (2.2,2.8);
		\coordinate[label=below:$n \leq -(m+3)$] (origin) at (2.7,2.2);
		
		% Modes on the diagonal $k=-n$ 
		%		\draw[black] (15,4) circle (5pt);  % g_{0,0} =0
		\draw[black] (13,6) circle (5pt);  
		\draw[black] (11,8) circle (5pt);  
		
		% Rectangle
		\node at (15,5)[draw,rectangle,minimum height=0.1cm]{};
		\node at (13,5)[draw, rectangle,minimum height=0.1cm]{}; 
		\node at (15,3)[draw,rectangle,minimum height=0.1cm]{};
		
		%		\node at (15,5)[draw,rectangle,minimum height=0.1cm]{};
		%		\node at (13,5)[draw, rectangle,minimum height=0.1cm]{}; 
		%		
		%% Traingle 
		
		\node[isosceles triangle, isosceles triangle apex angle=60,
		draw, rotate=90,  
		minimum size =0.01cm] (Trlast)at (11,9){};
		
		\node[isosceles triangle, isosceles triangle apex angle=60,
		draw, rotate=90,  
		minimum size =0.01cm] (Trlast1)at (9,9){};
		
		%%% RED FILLED FIRST
		\node[isosceles triangle, isosceles triangle apex angle=60,
		draw, rotate=270, fill=red!50,
		minimum size =0.01cm] (TR2a)at (15,0){};
		
		\node[isosceles triangle, isosceles triangle apex angle=60,
		draw, rotate=270,  fill=red!50,
		minimum size =0.01cm] (TR2b) at (15,8){};
		
		\node[isosceles triangle, isosceles triangle apex angle=60,
		draw, rotate=270,  fill=red!50,
		minimum size =0.01cm] (TR2b) at (7,8){};
		
		%%% RED FILLED SECOND
		\node[isosceles triangle, isosceles triangle apex angle=60,
		draw, rotate=240, fill=red!50,
		minimum size =0.01cm] (T3)at (13,3){};
		
		\node[isosceles triangle, isosceles triangle apex angle=60,
		draw, rotate=240,  fill=red!50,
		minimum size =0.01cm] (T4)at (11,5){};
		
		%%% RED FILLED LAST
		\node[isosceles triangle, isosceles triangle apex angle=60,
		draw, rotate=90,  fill=red!50,
		minimum size =0.01cm] (Trlast)at (7,3){};
		
		\node[isosceles triangle, isosceles triangle apex angle=60,
		draw, rotate=90,  fill=red!50,
		minimum size =0.01cm] (Trlast1)at (5,5){};
		
		%%%%% Trapezoidal
		%		\node at (15,7)[draw,trapezium,minimum height=0.1cm]{}; 
		%		\node at (9,7)[draw,trapezium,minimum height=0.1cm]{}; % g_{-3,1} is at (3,5) node
		%		\node at (15,1)[draw,trapezium,minimum height=0.1cm]{}; 
		
		%%% RED FILLED  FIRST
		\node at (15,1)[draw,fill=red!50,trapezium,minimum height=0.1cm]{}; 
		\node at (9,7)[draw,fill=red!50,trapezium,minimum height=0.1cm]{}; 		
		
		%%% RED FILLED SECOND
		\node at (3,9)[draw,fill=red!50,trapezium,minimum height=0.1cm]{}; 	
		\node at (15,7)[draw,fill=red!50,trapezium,minimum height=0.1cm]{};

		%%% Ellipse
		\draw[black] (13,7) ellipse (3pt and 6pt);
		\draw[black] (11,7) ellipse (3pt and 6pt);
		
		%		\draw[black] (15,9) ellipse (6pt and 3pt);
		%		\draw[black] (5,9) ellipse (6pt and 3pt);
		
		\draw[black] (13,9) ellipse (6pt and 8pt);
		\draw[black] (7,9) ellipse (6pt and 8pt);
		
		%%% RED FILLED FIRST		
		\draw[black,fill=red!50] (15,3) ellipse (9pt and 7pt);
		\draw[black,fill=red!50] (13,5) ellipse (9pt and 7pt);
		\draw[black,fill=red!50] (15,5) ellipse (9pt and 7pt);
		\draw[black,fill=red!50]  (5,9) ellipse (3pt and 6pt);
		\draw[black,fill=red!50]  (15,9) ellipse (3pt and 6pt);
		
		%%% RED FILLED SECOND
		\draw[black,fill=red!50] (13,0) ellipse (6pt and 8pt);
		\draw[black,fill=red!50,] (5,8) ellipse (6pt and 8pt);
		
		%%% RED FILLED LAST
		\draw[black,fill=red!50] (7,2) ellipse (8pt and 6pt);
		\draw[black,fill=red!50,] (3,6) ellipse (8pt and 6pt);
		
		%%%%%%%%%%%%%%%%%
		
		%		\node [rectangle, draw, xslant=0.4] at (15,8) {};
		%		\node [rectangle, draw, xslant=0.4] at (7,8) {};
		%		\node [rectangle, draw, xslant=0.4] at (15,0) {};
		
		%%% RED FILLED FIRST
		\node [rectangle, draw, fill=red!50, xslant=0.4] at (15,2) {};
		\node [rectangle, draw, fill=red!50, xslant=0.4] at (15,6) {};
		\node [rectangle, draw, fill=red!50, xslant=0.4] at (11,6) {};
		
		%%% RED FILLED SECOND
		\node [rectangle, draw, fill=red!50, xslant=-0.3] at (13,1) {};
		\node [rectangle, draw, fill=red!50, xslant=-0.3] at (7,7) {};

		%%% RED FILLED LAST
		%		\node [rectangle, draw, fill=red!50, xslant=0.6] at (0,7) {};
		\node [rectangle, draw, fill=red!50, xslant=0.6] at (7,0) {};

		\node[regular polygon, draw,
		regular polygon sides = 5] (p) at (13,8) {};
		
		\node[regular polygon, draw,
		regular polygon sides = 5] (p) at (9,8) {};
		
		%%% RED FILLED LAST
		\node[regular polygon, 	draw, fill=red!50,
		regular polygon sides = 5] (p) at (7,1) {};
		
		\node[regular polygon, draw, fill=red!50,
		regular polygon sides = 5] (p) at (1,7) {};

		%%% RED FILLED FIRST
		\node[regular polygon, 	draw, fill=red!50,
		regular polygon sides = 6] (p) at (13,2) {};
		
		\node[regular polygon, draw, fill=red!50,
		regular polygon sides = 6] (p) at (9,6) {};

		%		\foreach \x in {.5,1.5,...,8.5} {
		%			%\draw [thin] (\x,0) -- (\x,-0.1);
		%		}
		%		\foreach \y in {.5,1.5,...,8.5} {
		%			%\draw [thin] (0,\y) -- (-0.1,\y);
		%		}
		%\caption{l}
	\end{tikzpicture}
	\caption{An odd order $m$-tensor field $\bbf$ is determined by the even negative angular modes on or above the diagonal $k=-n$ (green region), and  the even negative angular modes (marked red) on the $\frac{m+1}{2}$ red lines $n+2k=-(m+1)$ for $k\geq0$. 
	All the even non-positive angular modes on and below the line $n+2k=-(m+3)$, and left of the line $n =-(m+3)$   vanish. 
		%All the even non-positive angular modes in the white region vanish.	
		\label{lattice_Oddtensor}
	}
\end{figure}
	
	The following result is a direct consequence of the algebraic  interaction of the range conditions in \eqref{RTCond_even_Oddtensor}, \eqref{RTCond_reality_mOdd}, 
	\eqref{RT_FourierEven_Oddtensor}, and \eqref{RTCond_Oddtensor}. To illustrate the result of this interaction,  let us consider the partition of $\ds \BZ^-\times\BZ$ as in  Figure \ref{lattice_Oddtensor}: 
	
	%%%%%%%%%%%%%%%%%%%%%%%%%%%%%%%
		For an odd integer $m \geq 1$,  the white region $W = W^+ \cup W^-$, where
		\begin{equation}\label{White_region_Oddtensor}
		\begin{aligned}
			W^+ & := \left \{ (n,k) \in \BZ^{-} \times \BZ^+  \; : \text{even } n\leq -m-3, \; \text{and} \; 0 \leq k \leq - \frac{n+m+3}{2}            \right \},   \\
			W^- & := \left \{ (n,k) \in \BZ^{-} \times \BZ^-  \; : \text{even } n\leq -m-3, \; \text{and} \;  k \leq 0            \right \},
		\end{aligned}
	\end{equation}
	the green region $G=G_L \cup G_R$, where
	\begin{equation}\label{Green_region_Oddtensor}
		\begin{aligned}
			G_L & := \left \{ (n,k) \in \BZ^{-} \times \BZ^+  \; : \text{even } n\leq -2, \; \text{and} \;     -\frac{n}{2}  \leq k \leq -n          \right \}.\\
			G_R & := \left \{ (n,k) \in \BZ^{-} \times \BZ^+  \; : \text{even } n\leq 0, \; \text{and} \;   k > -n            \right \}.
		\end{aligned}
	\end{equation}
	and the  red  region $R = R^+ \cup R^-$, where 
	\begin{equation}\label{Red_region_Oddtensor}
		\begin{aligned}
			R^+ & := \left \{ (n,k) \in \BZ^{-} \times \BZ^+  \; : \text{even } n\leq -2, \; \text{and} \;  - \frac{n+m+1}{2}  \leq k \leq - \frac{n}{2}            \right \},   \\
			R^- & := \left \{ (n,k) \in \BZ^{-} \times \BZ^-  \; : n \text{ even},   -m-1 \leq n\leq 0, \; \text{and} \;  k \leq 0            \right \}.
		\end{aligned}
	\end{equation}
	Note that the slanted line $\{(n,k): n+2k=0\}$ belongs to $G_L \cap R^+$. 
	
	\begin{remark}\label{InteractionODD}
	The modes in the region $R$ are affected solely by the symmetry \eqref{RT_FourierEven_Oddtensor} due to the double parametrization of the lines in $\Gam_+$. The modes in the region $G$  are affected both by the symmetry and the reality of the tensor  \eqref{RTCond_reality_mOdd}. The white region contains the modes affected by the nature of the operator (integration) along the line \eqref{RTCond_Oddtensor} in combination with the symmetry. See  Figure \ref{lattice_Oddtensor}.
\end{remark}

	The following result is a direct consequence of the algebraic interaction of the range
	\begin{cor}\label{corollary_Oddtensor}
		(i) Let $\bbf \in 	L^1(\mathbf{S}^m; \OM) $ be a real valued, integrable symmetric tensor field of odd order $m = 2q+1$, and  $g\in L^1_{\emph{skew,even}}(\Gam\times\sph)$ coincide with $X \bbf$ on $\Gam_+$, and $\{g_{n,k}\}$ be its Fourier coefficients.  Then, for all even $n\leq 0$,
		\begin{equation}\label{RTCond2neg}
			g_{n,k}=\left\{
			\begin{array}{ll}
				0,& \text{ if }
				(n,k) \in W,\\
				(-1)^{k}\overline{g_{-n-2k,k}},& \text{ if }
				(n,k) \in G,\\
				(-1)^{k}g_{n+2k,-k},& \text{ if }
				(n,k) \in R,\\
			\end{array}
			\right.
		\end{equation}see Figure \ref{lattice_Oddtensor}.

		(ii) Let $\{g_{n,k}\}$ be given for $(n,k) \in R^+\cup G_L$, such that 
		\begin{align}\label{gnk_decay_Oddtensor}
			\sum_{\substack{	(n,k) \in R^+ \cup G_L}} \jpn^{2} \lvert g_{n,k} \rvert < \INF, \quad \text{and} \quad
			\sum_{\substack{	(n,k) \in R^+ \cup G_L}}  \jpk^{1+\mu} 	\lvert g_{n,k} \rvert < \INF, %\quad 1/2 <\mu,
		\end{align}for some $\mu>1/2$. Extend $\{g_{n,k}\}$ from $(n,k) \in R^+\cup G_L$ to $R \cup G$ via \eqref{RTCond2neg}. Then there exists a real valued $\bbf \in L^1(\mathbf{S}^m; \OM)$, such that the mapping %$\{g_{-n,k}\}_{n, k \in \BZ}$ with
		\begin{align}\label{eq:gnk_map}
			(\Gamma\times\sph) \ni (e^{i \beta}, e^{i\theta}) \longmapsto  2\re\left\{  \sum_{\substack{	(n,k) \in R \cup G}} g_{n,k} e^{i n \theta}  e^{i k \beta}\right\}
			%			+\sum_{\substack{n\geq 1\\
			%			n  = \,\text{odd}}}  \sum_{k \in \BZ} \overline{g_{n,k}} e^{i n \theta}  e^{i k \beta},
		\end{align} 
		%defines a function $g^*(e^{i \beta}, e^{i\theta}) $, which satisfies \eqref{Xray_defn}.
		is precisely $X \bbf$ on $\Gam_+$, and $-X \bbf$ on $\Gam_-$. Moreover, $\bbf$  is uniquely determined by an element
		$ \left(\psi_0,  \psi_{-2},  \cdots ,\psi_{-2q}\right)  \in \Psi_{g}^{\emph{odd}} $.
	\end{cor}

	\section{$L^2$-analytic maps and their trace characterization} \label{sec:L2map}
	The method of proof of Theorems \ref{RangeCharac_Eventensor} and \ref{RangeCharac_Oddtensor} is based on the characterization in \cite{sadiqTamasan01}
	of traces of $A$-analytic maps in the sense of Bukhgeim \cite{bukhgeimBook}. In this section we summarize those existing results used in the proof.
%	These results hold for $\OM$ a strictly convex domain, not necessarily the unit disk.
	
	%In two dimensions, any symmetric $m$-covariant tensor may be represented as 
	%\begin{align*}
	%	{\bf f} = \sum_{k=0}^{m} {m \choose k} f_{k}({\bx}) \sigma (dx^{\otimes k} \otimes dy^{\otimes (m-k)})
	%\end{align*}
	%with $\sigma$ denoting the symmetrization operator \cite{vladimirBook}.

%	\newpage
%	\commentK{Below are some results for $0$-tensors.}
%	For a real valued $u(z,\btheta)$, $u_{-n}=\ol{u_n}$ and the angular dependence is completely determined by the sequence $\bu$ of its nonpositive Fourier modes,
%	\begin{align}\label{boldu}
%		\OM \ni z\mapsto  \bu(z)&: = \langle u_{0}(z), u_{-1}(z),u_{-2}(z),... \rangle.
%	\end{align}
%	In particular, $\bu$ solves the Beltrami-like equation
%	\begin{align}\label{beltrami}
%		\dba\bu(z) +L^2 \del\bu(z) = 0,\quad z\in \OM,
%	\end{align}
%	where  $L\bu(z)=L (u_0(z),u_{-1}(z),u_{-2}(z),...):=(u_{-1}(z),u_{-2}(z),...)$ denotes the left translation.
	
	%Our data $u \lvert_{\Gam \times \sph}$ yields the trace of  a solution of \eqref{beltrami} on the boundary,  
	%\begin{align}\label{gdata_defn}
	% \bg = \bu \lvert_{\Gam} = \langle g_{0}, g_{-1}, g_{-2},... \rangle.
	%\end{align}
	
	Bukhgeim's original  theory in \cite{bukhgeimBook}  considers the sequence valued maps 
		\begin{align}\label{boldu}
		\OM \ni z\mapsto  \bu(z)&: = \langle u_{0}(z), u_{-1}(z),u_{-2}(z),... \rangle,
	\end{align} and solution of the Beltrami-like equation 
    \begin{align}\label{beltrami}
    	\dba\bu(z) + \LL \del\bu(z) = 0,\quad z\in \OM,
    \end{align}
    where  $\LL\bu(z)= \LL (u_0(z),u_{-1}(z),u_{-2}(z),...):=(u_{-1}(z),u_{-2}(z),...)$ denotes the left translation.
%    , and $\LL^{m}=\LL\circ \cdots \circ \LL$ be its $m$-th composition. %, $k=1,2$.
	These solutions are called $\LL$-analytic.%, $k=1,2$.

	Similar to classical analytic maps, the solution of \eqref{beltrami} satisfy a Cauchy-like integral formula,
	\begin{align}\label{Analytic}
		\bu (z) = \B [\bu \lvert_{\Gam}](z), \quad  z\in\OM,
	\end{align} where $\B$ is the Bukhgeim-Cauchy operator  acting on $\bu \lvert_{\Gam}$ defined component-wise \cite{finch} for $n\geq 0$ by
	\begin{align} \label{BukhgeimCauchyFormula}
		% u_{-n}(z) &=  (\B \bg)_{-n}(z) \\
		(\B \bu)_{-n}(z) &:= \frac{1}{2\pi i} \int_{\Gam}
		\frac{ u_{-n}(\zeta)}{\zeta-z}d\zeta  + \frac{1}{2\pi i}\int_{\Gam} \left \{ \frac{d\zeta}{\zeta-z}-\frac{d \ol{\zeta}}{\ol{\zeta}-\ol{z}} \right \} \sum_{j=1}^{\infty}  
		u_{-n-j}(\zeta)
		\left( \frac{\ol{\zeta}-\ol{z}}{\zeta-z} \right) ^{j},\; z\in\OM.
	\end{align}
	% As shown in \cite[Theorem 2.2]{sadiqTamasan02}, if $\bu=\langle u_{0}, u_{-1}, u_{-2},...\rangle  \in l^{1,1}_{\INF}(\del \OM)\cap C^\mu(\del \OM;l_1)$, then $\B \bu\in C^{1,\mu}(\OM;l_\infty)\cap C(\ol \OM;l_\infty)$  is $L^2$-analytic in $\OM$. If  $\bu\in C^{1,\mu}(\OM;l_\infty)\cap C(\ol \OM;l_\infty)$ is $L^2$-analytic in $\ol\OM$, then $\bu(z)=\B\bu(z)$, for $z\in\ol\OM$. 
	
	% Similar to the  classical analytic maps, traces of the boundary of solutions of
	% \eqref{Analytic} are constrained. In \cite{sadiqTamasan01} the Bukhgeim-Hilbert transform 
	% \begin{align}\label{BHtransform}
	% (\HT \bg)_{-n}(\zeta)=\frac{1}{\pi }\int_\Gam\frac{g_{-n}(z)}{z-\zeta}dz + 
	% \frac{1}{\pi } \int_\Gam \left\{
	% \frac{dz}{z-\zeta} - \frac{d\ol{z}}{\ol{z}-\ol{\zeta}}
	% \right\}\sum_{j=1}^\infty g_{-n-2j}(z)\left(\frac{\ol{z-\zeta}}{z-\zeta}\right)^j,\zeta\in\Gam,
	% \end{align}
	% was introduced, and shown that
	% \begin{align}\label{sadiqTamasan_characterization} %\cite[Theorem 3.2]{sadiqTamasan01}
	% \text{$\bu$ is $L^2$-analytic } & \iff  \quad  \text{its trace} \; \,  \bu|_\Gam\in Ker[I+i\HT],
	% \end{align}
	% % $\bu$ is $L^2$-analytic if and only if its trace $\bu|_\Gam\in Ker[I+i\HT]$, 
	% see \cite[Theorem 3.2]{sadiqTamasan01} for analytical details. 

	The traces of $\LL$-analytic maps  on the boundary must satisfy some constraints, which can be expressed in terms of a corresponding Hilbert-like transform introduced in  \cite{sadiqTamasan01}. More precisely, the Bukhgeim-Hilbert transform $\HT$ acting on  $\bg$, 
	\begin{align}\label{boldHg}
		\Gam \ni z\mapsto  (\HT \bg)(z)& = \langle (\HT \bg)_{0}(z), (\HT \bg)_{-1}(z),(\HT \bg)_{-2}(z),... \rangle
	\end{align}
	is defined component-wise for $n\geq 0$ by
	\begin{align}\label{BHtransform}
		(\HT \bg)_{-n}(z)=\frac{1}{\pi }\int_\Gam \frac{ g_{-n}(\zeta)}{\zeta-z}d\zeta  + \frac{1}{\pi }\int_{\Gam} \left \{ \frac{d\zeta}{\zeta-z}-\frac{d \ol{\zeta}}{\ol{\zeta}-\ol{z}} \right \} \sum_{j=1}^{\infty}  
		g_{-n-j}(\zeta)
		\left( \frac{\ol{\zeta}-\ol{z}}{\zeta-z} \right) ^{j}, z\in\Gam.
	\end{align}
	
	The theorems below comprise some results in  \cite{sadiqTamasan01,sadiqTamasan02,sadiqTamasan22}. For $0<\mu<1$, $p=1,2$, we consider the  Banach spaces:
	\begin{equation}\label{spaces}
	\begin{aligned} 
			l^{1,p}_{\INF}(\Gam) &:= \left \{ \bg= \langle g_{0}, g_{-1}, g_{-2},...\rangle\; : \lnorm{\bg}_{l^{1,p}_{\INF}(\Gam)}:= \sup_{\xi \in \Gam}\sum_{j=0}^{\INF}  \jpj^p \lvert g_{-j}(\xi) \rvert < \INF \right \},\\
			C^{\mu}(\Gam; l_1) &:= \left \{ \bg= \langle g_{0}, g_{-1}, g_{-2},...\rangle:
			\sup_{\xi\in \Gam} \lVert \bg(\xi)\rVert_{\ds l_{1}} + \underset{{\substack{
						\xi,\eta \in \Gam \\
						\xi\neq \eta } }}{\sup}
			\frac{\lVert \bg(\xi) - \bg(\eta)\rVert_{\ds l_{1}}}{|\xi - \eta|^{ \mu}} < \INF \right \}, \\
			Y_{\mu}(\Gam) &:= \left \{ \bg: \bg \in  l^{1,2}_{\INF}(\Gam) \; \text{and} \;
			\underset{{\substack{
						\xi,\eta \in \Gam \\
						\xi\neq \eta } }}{\sup} \sum_{j=0}^{\INF}  \jpj 
			\frac{\lvert g_{-j}(\xi) - g_{-j}(\eta)\rvert }{|\xi - \eta|^{ \mu}} < \INF \right \},
		\end{aligned}
	\end{equation} where, for brevity, we use the notation $\jpj=(1+|j|^2)^{1/2}$.
	Similarly,  we consider $ C^{\mu}(\ol \OM; l_\INF) $, and $  C^{\mu}(\OM; l_\INF) = \bigcup_{0<r<1} C^{\mu}(\ol \OM_r; l_\INF)$, where for $0<r<1$, $\OM_r = \{ z \in \BC : |z| <r  \}$.
	% The following result recalls the necessary and sufficient conditions for a sufficiently regular map to be the boundary value of an $L^2$-analytic function.

	%%%%%%%%%%%%%%%%%%%%%%%%%%%%%%%%%%%%%
	\begin{theorem}\label{BukhgeimCauchyThm}
		Let $0<\mu<1$. Let $\bg = \langle g_{0}, g_{-1}, g_{-2},...\rangle$ be a sequence valued map defined on the boundary $\Gam$ and $\B$ be the Bukhgeim-Cauchy operator acting on $\bg$ as in \eqref{BukhgeimCauchyFormula}. 
		
		(i) If $\bg \in l^{1,1}_{\INF}(\Gam)\cap C^\mu(\Gam;l_1)$, then $\bu := \B \bg\in C^{1,\mu}(\OM;l_\infty)\cap C(\ol \OM;l_\infty)$ is $\LL$-analytic in $\OM$.
		
		(ii) Moreover, if $\bg\in Y_{\mu}(\Gam)$ for $\mu>1/2$, then $ \B \bg \in C^{1,\mu}(\OM;l_1)\cap C^{\mu}(\ol \OM;l_1)$. 
	\end{theorem}For the proof of Theorem \ref {BukhgeimCauchyThm} (i) we refer to \cite[Theorem 3.1]{sadiqTamasan01}, and for part (ii) we refer to \cite[Proposition 2.3]{sadiqTamasan02}.
	
	%\begin{theorem} \label{Bukh_AanalyticEq_Thm}
	%Let $0<\mu<1$. If  $\bu\in C^{1,\mu}(\OM;l_\infty)\cap C(\ol \OM;l_\infty)$ is $L^2$-analytic in $\ol\OM$, then  $\bu(z)=\B [\bu \lvert_{\Gam}](z)$, for $z\in\ol\OM$. 
	%\end{theorem}
	
	The following result characterize the traces of $\LL$-analytic maps.
	\begin{theorem}\label{NecSuf_BukhgeimHilbert_Thm}
		
		Let $0<\mu<1$, and let $\HT$ be the Bukhgeim-Hilbert transform in \eqref{BHtransform}.
		%
		%(i) If $\bg \in l^{1,1}_{\INF}(\Gam)\cap C^\mu(\Gam;l_1)$, then $\HT \bg\in C^{\mu}(\Gam;l_\infty)$. 
		%
		%(ii) If $\bg\in Y_{\mu}(\Gam)$ for $\mu>1/2$, then $ \HT \bg\in C^{\mu}(\Gam;l_1)$.
		
		(i) If $\bg \in l^{1,1}_{\INF}(\Gam)\cap C^\mu(\Gam;l_1)$ is the boundary value of an $\LL$-analytic function,
		then $\HT \bg\in C^{\mu}(\Gam;l_\infty)$ and satisfies 
		\begin{align} \label{NecSufEq}
			(I+ i \HT) \bg = {\bf {0}}.
		\end{align}
		(ii)  If $\bg\in Y_{\mu}(\Gam)$ for $\mu>1/2$,  satisfies \eqref{NecSufEq}, then $\LL$-analytic function
		 $\bu := \B \bg \in C^{1,\mu}(\OM;l_1)\cap C^{\mu}(\ol \OM;l_1)$, satisfy
		\begin{align}\label{gdata_defn}
			\bu \lvert_{\Gam} = \bg.
		\end{align}
		%
		% it is necessary and sufficient that
		%\begin{align} \label{NecSufEq}
		%   (I+ i \HT) \bg = {\bf {0}}.
		%\end{align} %where $\HT$ is the Bukhgeim-Hilbert transform in \eqref{BHtransform}.
	\end{theorem}
	For the proof  of Theorem \ref{NecSuf_BukhgeimHilbert_Thm}  we refer to \cite[Proposition 3.1, Theorem 3.2,  Corollary 4.1, and Proposition 4.2]{sadiqTamasan01}.
	%For the proof of the last statement of the Theorem we refer to \cite[]{sadiqTamasan01}.
	
	The results above  need $\OM$ be a strictly convex domain, but not necessarily a unit disk.
	However, the following result \cite[Theorem 4.1]{sadiqTamasan22} uses  $\Omega$ be the unit disk, and $\Gam$ be its unit circle boundary.
%	\commentK{Need to state the Theorems for $\mathcal{L}$ analytic maps.}
	Given $\bg = \langle g_0, g_{-1}, g_{-2}, ... \rangle \in l_{\INF}(\BN; L^1(\Gam))$, we consider the Fourier coefficients of its components 
	\begin{align}\label{gnk_bg}
		g_{-n,k} := \frac{1}{2 \pi} \int_{-\pi}^{ \pi} g_{-n} \left( e^{i  \beta} \right) e^{-i k \beta} d \beta, \text{ for all } n \geq 0, \text{ and } k \in \BZ.
	\end{align}

	%\begin{remark}\label{defn:bg_gnk}
	%If $\ds \sup_{n\geq 0} \sum_{k=-\infty}^\infty \lvert g_{-n,k} \rvert < \INF$, 
	%then 
	%% where  for each $n \geq 0$, %for $\Gam \ni z =e^{i  \beta}$, and
	%\begin{align}\label{gnk_bg1}
	%g_{-n}(e^{i  \beta})=\sum_{k=-\infty}^\infty g_{-n,k} \, e^{i k \beta}, \;\forall n \geq 0, \; \text{defines} \quad 
	%\bg = \langle g_0, g_{-1}, g_{-2}, ... \rangle \in l_{\INF}(\BN; C^1(\Gam)).
	%\end{align} 
	%\end{remark}
	%%%%%%%%%%%%%%%%%%%%%%%%%%%%%%%%%%%%%%%%%%%%%%%%%%%%
	
	\begin{theorem}\cite[Theorem 4.1]{sadiqTamasan22} \label{newmapping_Hilbert}
		Let $\bg=  \langle g_{0},g_{-1}, g_{-2} ... \rangle \in l^{1,1}_{\INF}(\Gam)\cap C^\mu(\Gam;l_1)$, $0< \mu <1$, and $g_{-n,k}$ be the Fourier coefficients of its components as in \eqref{gnk_bg}.
		Let $\ds \HT \bg = \langle (\HT \bg)_{0}, (\HT \bg)_{-1},(\HT \bg)_{-2},... \rangle$ be the Bukhgeim-Hilbert transform  acting on $\bg$ as defined in \eqref{BHtransform}.
		Then  $\HT \bg\in C^{\mu}(\Gam;l_\infty)$, and the Fourier coefficients 
		$\ds (\HT \bg)_{-n,k} := \frac{1}{2 \pi} \int_{-\pi}^{ \pi} (\HT \bg)_{-n} \left( e^{i  \beta} \right) e^{-i k \beta} d \beta$, for $ n \geq 0,  k \in \BZ, $
		of its components satisfy 
		\begin{align}\label{Hu_comp}
			(-i) (\HT \bg)_{-n,k}=\left\{
			\begin{array}{ll}g_{-n,k}&\mbox{ if } k\geq 0,\\
				-g_{-n,k}+ 2(-1)^k g_{-n+2k,-k}&\mbox{ if } k\leq -1.
			\end{array}
			\right.\end{align}
	\end{theorem}

\section{ Proof of Theorems \ref{RangeCharac_Eventensor} and \ref{RangeCharac_Oddtensor}}\label{Sec:pf_mainTh}
	
	%	\section{ Proof of the Theorem \ref{RangeCharac_Eventensor} }\label{Sec:pf_mainTh}

	Since $\bbf$ is symmetric,  %(i.e, $ f_{i_1 \cdots i_m}= f_{ i_{\sigma(1)} \cdots i_{\sigma(m)}}$ for any $m$-permutation $\sigma$), 
	for any $m$-tuple $(i_1, \cdots, i_m)\in \{1,2\}^m$ such that $2$ occurs exactly $k$ times (and $1$ occurs $m-k$ times), the component $f_{i_1...i_m}$ satisfies
	\begin{align}\label{eq:cov_tensor}
		f_{i_1...i_m}= f_{ \underbrace{1\cdots1}_{m-k} \underbrace{2\cdots2}_{k}} =:\tilde{f}_k.
	\end{align}
	Since there are ${m \choose k}$ many $m$-tuples $(i_1,i_2, \cdots, i_m)$ that contain exactly $k$ many $2's$, we get
		\begin{align*}\label{eq:innerprod_fThetaM}
			\langle \bbf(z), \btheta^m \rangle &= %\sum_{(i_1,i_2, \cdots, i_m)\in \{1,2\}^m} 
			f_{i_1 \cdots i_m} (z) \theta^{i_1} \cdot \theta^{i_2} \cdots \theta^{i_m} 
			\xlongequal{  \text{}\eqref{eq:cov_tensor} } 
			%\sum_{k=0}^{m} {m \choose k} \tilde{f}_k (x) \;  (\theta^1)^{ (m-k)}  (\theta^2)^{ k} =
			\sum_{k=0}^{m} {m \choose k} \; \tilde{f}_k (z)  \;  (\cos \theta)^{ (m-k)}  (\sin \theta)^{ k} \\
			&= (e^{-\i \tta} )^m  \sum_{k=0}^{m} \frac{(-\i)^k}{2^m}  {m \choose k} \tilde{f}_k (z)   \;  Q_{m,k} (e^{2\i \tta}),
			% \sum_{k=0}^{m} {m \choose k} \; \tilde{f}_k (x)  \; \left (\frac{e^{\i \tta}+e^{-\i\tta}}{2}\right)^{ m-k}  \left (\frac{e^{\i \tta}-e^{-\i\tta}}{2 \i}\right)^{ k} \\
			%&= e^{-\i m\tta} \sum_{k=0}^{m} \frac{(-\i)^k}{2^m} {m \choose k} \; \tilde{f}_k (x)   \; \left (e^{2\i \tta}+1 \right)^{ m-k}  \left (e^{2\i \tta}-1\right)^{ k} \\
			%&= (e^{-\i \tta} )^m \sum_{k=0}^{m} {q}_k (x)   \; \left ( (e^{\i \tta})^2+1 \right)^{ m-k}  \left ( (e^{\i \tta})^2-1\right)^{ k} \\
			%&\xlongequal{  \text{where } t=e^{\i \tta} }  t^{-m} \sum_{k=0}^{m} {g}_k (x)   \; Q_{m,k} (t^2), \\ %\quad \text{where } t=e^{\i \tta} ,   \\
			%&\xlongequal{  \text{By Lemma } \ref{Qlemma} } t^{-m} \sum_{k=0}^{m} {h}_k (x)   (t^2)^k 
			%=  \sum_{k=0}^{m}  \; h_k (x) e^{\i (2k-m)\tta},
			%&= %\sum_{k=0}^{m} {m \choose k} \; \bF_{k} e^{\i (2k-m)\varphi} =
			%\begin{cases} 
			%%   \ds \sum_{k=-\frac{m}{2}}^{\frac{m}{2}} f_{2k} e^{\i (2k)\varphi},  & (m \;\text{even}), \\
			%%   \ds \sum_{k=-\frac{m-1}{2}}^{\frac{m-1}{2}} f_{2k-1} e^{\i (2k-1)\varphi}, & (m \;\text{odd}),
			%% \ds \sum_{k=-q}^{q} f_{-2k} e^{\i (2k)\tta},   & (m =2q,\; q \geq 0), \\
			%\ds \sum_{k=0}^{q} f_{2k} e^{-\i (2k)\tta} + \sum_{k=1}^{q} f_{-2k} e^{\i (2k)\tta},   & (m =2q,\; q \geq 0), \\
			%% \ds \sum_{k=-q-1}^{q} f_{-(2k+1)} e^{\i (2k+1)\tta},  & (m =2q+1,\; q \geq 0),
			%\ds \sum_{k=0}^{q} f_{2k+1} e^{-\i (2k+1)\tta} + f_{-(2k+1)} e^{\i (2k+1)\tta},  & (m =2q+1,\; q \geq 0),
			%\end{cases}
		\end{align*} where  $Q_{m.k} (t) =  \left( t+ 1 \right)^{m-k} \left( t - 1 \right)^k$.
	
	Since $\ds \{Q_{m,k}(t)\}_{k=0}^{m}$ form a basis for the space of polynomials of degree $m$,
	\begin{equation}\label{eq:innerprod_fThetaM}
		\begin{aligned}
			\langle \bbf, \btheta^m \rangle   &=
			\begin{cases} 
				\ds \sum_{k=0}^{q} f_{2k} e^{-\i (2k)\tta} + \sum_{k=1}^{q} f_{-2k} e^{\i (2k)\tta},   &\mbox{ if  } m =2q,\\
				\ds \sum_{k=0}^{q} f_{2k+1} e^{-\i (2k+1)\tta} + f_{-(2k+1)} e^{\i (2k+1)\tta},  &\mbox{ if  } m =2q+1,
			\end{cases}
		\end{aligned}
	\end{equation} %where $ \ol{f_{n}} = f_{-n}, \; -m \leq n \leq m, \, m \geq 0$
	where $f_k $'s are in a one -to-one correspondence to $\tilde{f}_k$, and thus with $f_{i_1 \cdots i_m} $. We refer to the Lemma  \ref{Qlemma} in the appendix for details on this one-to-one correspondence in \eqref{eq:innerprod_fThetaM}.

	We approach the range characterization via the well-known connection with the transport model, where the unique solution $u(z,\btheta)$ to the boundary value problem
	\begin{subequations}\label{bvp_transport}
		\begin{align}\label{TransportEq1}
			\btheta\cdot\nabla u(z,\btheta) &=  2\langle  \bbf(z), \btheta^m \rangle
			\\  \label{u_Gam-}
			u \lvert_{\Gam_{-}} &= - X \bbf
		\end{align}
	\end{subequations}
	has the trace $u\lvert_{\Gam\times\sph} = g $, with $g$ in \eqref{eq:g_mEven}, i.e.
%	\begin{align*}
%		u\lvert_{\Gam\times\sph} = g 
%		\end{align*}
%	and $u|_{\Gam_+}=  Xf|_{\Gam_+}$ on $\Gam_+$. 
%	The trace $u\lvert_{\Gam\times\sph}$ function is given by
	\begin{equation}\label{eq:gstar1}
		u\lvert_{\Gam\times\sph}=
		\left\{
		\begin{array}{ll}
			X \bbf ,&\mbox{ on } \Gam_+,\\
			- X \bbf, &\mbox{ on } \Gam_-.
		\end{array}
		\right.
	\end{equation} 
\begin{prop}\label{prop:Uoddeven_mEvenOdd}
	(a)  If $m = 2q, \, q\geq 0$, then the solution $u$ to  the boundary value problem \eqref{bvp_transport} is an odd function of $\btheta$, 
	       $$u(z,\btheta) = - u(z, -\btheta).$$
	(b)  If $m = 2q+1, \, q\geq 0$, then the solution $u$ to  the boundary value problem   \eqref{bvp_transport} is an even function of $\btheta$,    $$u(z,\btheta) =  u(z, -\btheta).$$    
\end{prop}
\begin{proof}
% (a) We note that $\langle  \bbf(z), (-\btheta)^m \rangle= (-1)^m\langle  \bbf(z), \btheta^m \rangle = \langle  \bbf(z), \btheta^m \rangle$, for even $m\geq0$.
 (a)  We note that for even $m \geq 0$, $(-\btheta)^m = \btheta^m$. \\
     If $u^{\text{odd}}(z,\btheta):=\frac{1}{2}\left[u(z,\btheta)-u(z,-\btheta)\right]$ denotes the angularly odd part of $u$, then 
	\begin{align*}%\label{TransportEqEven}
		\btheta\cdot\nabla u^{\text{odd}}(z,\btheta) = \frac{1}{2}\left[\btheta\cdot\nabla u(z,\btheta)+(-\btheta)\cdot\nabla u(z,-\btheta) \right]
   %	=  \frac{1}{2}\left[2\langle  \bbf(z), \btheta^m \rangle+2\langle  \bbf(z), \btheta^m \rangle \right]
    =2\langle  \bbf(z), \btheta^m \rangle,
	\end{align*}and $u^{\text{even}}:=u -u^{\text{odd}}$ solves
	\begin{subequations}\label{bvp_transport_even}
		\begin{align}%\label{TransportEqEven}
			\btheta\cdot\nabla u^{\text{even}}(z,\btheta) &=  0 , \quad (z,\btheta)\in \OM\times\sph, \\  %\label{u_Gam-}
			u^{\text{even}}\lvert_{\Gam_{-}} &=0.
		\end{align}
	\end{subequations}
	Since \eqref{bvp_transport_even} has the unique solution $u^{\text{even}}\equiv 0$, $u=u^{\text{odd}}$ in $\ol\OM\times\sph$. \\
	(b)   We note that for odd $m \geq 1$, $(-\btheta)^m =  - \btheta^m$. \\
	If $u^{\text{even}}(z,\btheta):=\frac{1}{2}\left[u(z,\btheta)+u(z,-\btheta)\right]$ denotes the angularly even part of $u$, then 
	\begin{align*}%\label{TransportEqEven}
		\btheta\cdot\nabla u^{\text{even}}(z,\btheta) = \frac{1}{2}\left[\btheta\cdot\nabla u(z,\btheta)
		-(-\btheta) \cdot\nabla u(z,-\btheta) \right]
		%	=  \frac{1}{2}\left[2\langle  \bbf(z), \btheta^m \rangle+2\langle  \bbf(z), \btheta^m \rangle \right]
		= 2\langle  \bbf(z), \btheta^m \rangle,
	\end{align*}and $u^{\text{odd}}:=u -u^{\text{even}}$ solves
	\begin{subequations}\label{bvp_transport_Modd}
		\begin{align}%\label{TransportEqEven}
			\btheta\cdot\nabla u^{\text{odd}}(z,\btheta) &=  0 , \quad (z,\btheta)\in \OM\times\sph, \\  %\label{u_Gam-}
			u^{\text{odd}}\lvert_{\Gam_{-}} &=0.
		\end{align}
	\end{subequations}
	Since \eqref{bvp_transport_Modd} has the unique solution $u^{\text{odd}}\equiv 0$, $u=u^{\text{even}}$ in $\ol\OM\times\sph$.
\end{proof}
 
% Note that regardless of the parity of $m$,  $u|_{\Gam_+}=  Xf|_{\Gam_+}$ on $\Gam_+$.   
 
 We present in detail the proof for the even tensor $m=2q$ (Theorem \ref{RangeCharac_Eventensor}).
 For the odd tensor case  (Theorem \ref{RangeCharac_Oddtensor}), the proof is essentially the same, where small changes occur due to the change in parity.

	(i) {\bf Necessity of Theorem \ref {RangeCharac_Eventensor}:} Since $g$ is angularly odd, \eqref{RTCond_odd_Eventensor} holds.
	Since $g$ is real valued, \eqref{RTCond_reality_mEven} holds. The identities \eqref{RT_FourierOdd_Eventensor} follow by direct calculation, see Lemma \ref{lem:evenness}.

	We will first prove \eqref{RTCond_Eventensor} for  an even order $m=2q$ tensor $\bbf= (f_{i_1i_2...i_m})$ with smooth components $f_{i_1i_2...i_m}  \in C^{2}_0(\OM)$. The result for components $f_{i_1i_2...i_m} \in L^1(\OM)$ follows by a density argument.
	
	By using  the notations $\dba = (\del_{x_1}+i\del_{x_2})/2$, $\del =(\del_{x_1}-i\del_{x_2})/2$, and $\theta=\arg{\btheta}\in (-\pi,\pi]$,
	%$\btheta=(\cos\tta,\sin\tta)$. 	% then $\btheta \cdot\nabla=e^{-i\tta}\dba + e^{i\tta}\del$.
	%For the real valued symmetric tensor field $\bbf$ of even order  $m=2q, \; q \geq 0$, 
the transport equation \eqref{TransportEq1} becomes 
	\begin{align}\label{TransportEq_mEven}
		&[e^{-i\tta}\dba + e^{i\tta}\del] u(z,\btheta)= \sum_{k=-q}^{q} f_{2n}(z) e^{- \i (2n) \tta}, \quad (z,\btheta)\in \OM\times\sph.
		%\sum_{k=0}^{q} f_{2k} e^{-\i (2k)\tta} + \sum_{k=1}^{q} f_{-2k} e^{\i (2k)\tta}, 
	\end{align}
	%Since $\bbf$ is real valued, $ \ol{f_{2n}} = f_{-2n}, \; -q \leq n\leq q$; in particular $f_0$ is real-valued.
	% while other modes are complex conjugates.

	 By Proposition \ref{prop:Uoddeven_mEvenOdd} (a),  the solution $u$ to \eqref{TransportEq_mEven} %the boundary value problem \eqref{bvp_transport} 
	 is an odd function of $\btheta$, thus 
%	$$u(z,\btheta) = - u(z, -\btheta).$$
	all of its Fourier coefficients (in the angular variable) of even order vanish,
	$$u(z,\btheta)=\sum_{\substack{n \in \BZ\\	n  =\,\text{odd}}}  u_{n}(z) e^{in\tta}.$$
	%	If $\ds \sum_{\substack{n \in \BZ\\
%			n  =\,\emph{odd}}}  u_{n}(z) e^{in\tta}$ is the Fourier series  expansion in the angular variable $\btheta$ of a solution $u$ of \eqref{TransportEq_mEven}, then, %Provided some sufficient decay (to be specified later) of $u_n$ to allow regrouping, and 
	By identifying the Fourier modes of the same order, the equation \eqref{TransportEq_mEven} reduces to the system:
	\begin{align} \label{NART_mEven_fmEq}
		&\ol{\del} u_{-(2n-1)}(z) + \del u_{-(2n+1)}(z) = f_{2n}(z), && 0 \leq n \leq q,\\ 
		\label{NARTmEvenAnalyticEq_Odd}
		&\ol{\del} u_{-(2n-1)}(z) + \del u_{-(2n+1)}(z) = 0, && n \geq q+1.
		%		\label{NARTmEvenAnalyticEq_Even}
		%		&\ol{\del} u_{-2n}(z) + \del u_{-(2n+2)}(z) = 0, &&  n \geq 0.
	\end{align}
	
	Since $\bbf$ is real valued, the solution $\ds u(z,\btheta) $ of \eqref{TransportEq_mEven} is also  real valued, and its Fourier modes in the angular variable occur in conjugates, $u_{-n}=\ol{u_n}$. Thus, it suffices to consider the non-positive odd Fourier modes of $u(z,\cdot)$: Let $\bu$ be  the sequence valued map
	\begin{align}\label{boldu1}
		\OM \ni z\mapsto  \bu(z)&: = \langle u_{-1}(z), u_{-3}(z),u_{-5}(z),... \rangle.
	\end{align}
	
	Since all the components $f_{i_1i_2...i_m}  \in C^{2}_0(\OM)$, $X\bbf\in C^2(\Gam_-)$, and, thus, the solution $u$ to the transport problem \eqref{TransportEq_mEven} is in $C^{2}(\ol \OM\times \sph)$. %In particular $\bu \in C^1(\ol \OM; l_1)$. 
	Moreover, its trace $u \lvert_{\Gam \times \sph}\in C^2(\Gam\times\sph)$.
	
	 %We define the sequence valued map 
	% Let $\bg$ be the sequence valued map generated by the negative, 
	 
	For $z=e^{\i\beta}\in\Gam$, we use the negative odd Fourier modes of the trace $g(e^{i\beta},\cdot)=u \lvert_{\Gam \times \sph}(e^{\i\beta},\cdot)$ to define the sequence valued map
	\begin{align}\label{eq:bg1}
	\bg(e^{i \beta}) := \langle g_{-1} (e^{i \beta}), g_{-3}(e^{i \beta}), ... \rangle, \mbox{ with } g_{-2n-1}(e^{i \beta}) = \frac{1}{2 \pi} \int_{-\pi}^{ \pi} g(e^{i \beta}e^{i\theta}) e^{i (2n+1)\theta} d\theta.
	\end{align}
	%is the $(-n)$-th Fourier coefficients in the angular variable of the function $g$.
	%Since $C^2(\Gam \times \sph) \subset C^\mu (\Gam;C^{1,\mu}(\sph))$, classical regularity results on the Fourier transform ensure $\bg^*  \in  l^{1,1}_{\INF}(\Gam) \cap C^{\mu}(\Gam; l^1)$, $\mu >1/2$, see \cite[Proposition 4.1 (i)]{sadiqTamasan01} for details.
	%Since  $ f \in C^{2}_0(\OM)$ and $\psi\in C^2_{sym}(\Gam_-)$, $g$ in \eqref{eq:gstar1} is in $ C^{2}(\Gam \times \sph)$.
	
	Since $u\in C^{2}(\ol \OM\times \sph)$, $\bu \in C^1(\ol \OM; l_1)$, and, thus,
	\begin{align}\label{trace}
	\bg = \bu \lvert_{\Gam} \, \in C^1(\Gam; l_1) \subset  l^{1,1}_{\INF}(\Gam) \cap C^{\mu}(\Gam; l^1), \mu >1/2.\end{align}
	
	By \eqref{NARTmEvenAnalyticEq_Odd}, 
	%\begin{align*}%\label{Lq}
	$\ds \LL^{q}\bu = 	\langle u_{-(2q+1)}, u_{-(2q+3)},  \cdots \rangle$
	%\end{align*} 
	is $\LL$-analytic in $\OM$  and its trace $\LL^{q} \bg = \LL^{q}  \bu \lvert_{\Gam}\ds$ is the boundary value of an $\LL$-analytic map. 
	
	Recall that $\HT$ is the Bukhgeim-Hilbert operator in \eqref{BHtransform}. By the necessity part in Theorem \ref{NecSuf_BukhgeimHilbert_Thm}, we have $\HT (\LL^{q} \bg) \in C^{\mu}(\Gam; l_\INF)$ and
	%\begin{align}\label{Hg_char}
	$\ds	(I+i\HT)  (\LL^{q} \bg)= \bzero.$
	%\end{align}
	
	Since $\HT$ commutes with the left translation $\LL$, we obtained
\begin{align}\label{Hg_char}
		\LL^{q} (I+i\HT)  (\bg)= \bzero.
	\end{align}
	In particular, for all odd $n \leq -2q-1$, and $k \in \BZ$, we obtained
	\begin{align}\label{Hgnk_char}
		([I+i\HT] \bg)_{n,k}= 0.
	\end{align}

	By Theorem \ref{newmapping_Hilbert}, the Fourier coefficients $\ds (\HT \bg)_{n,k}$, for odd $ n \leq -2q-1$ and  $k \in \BZ$,
	satisfy \eqref{Hu_comp}, 
	%\begin{align*}%\label{eq:BHT_gnk}
	%(-i) (\HT \bg)_{-n,k}=\left\{
	%\begin{array}{ll}g _{-n,k}&\mbox{ if } k\geq 0,\\
	%-g_{-n,k}+ 2(-1)^k g_{-n+2k,-k}&\mbox{ if } k\leq -1,
	%\end{array}
	%\right.\end{align*}
	and thus
	\begin{align*}%\label{eq:HT_gnk}
		([I+i\HT] \bg)_{n,k}=\left\{
		\begin{array}{ll} 0 &\mbox{ if } k\geq 0,\\
			2g_{n,k}-2 (-1)^k  g _{n+2k,-k}&\mbox{ if } k\leq -1.
		\end{array}
		\right.\end{align*}In conjunction with \eqref{Hgnk_char}, the Fourier coefficients of $\bg$ must satisfy \eqref{RTCond_Eventensor}, i.e., 
	\begin{align*}%\label{RTCond}
		g_{n,k}=  (-1)^k  g_{n+2k,-k}, \text{ for all odd } n\leq -2q-1, \text{ and } k\leq -1.
	\end{align*}
	Equation \eqref{RTCond_Eventensor} for $k=0$ is trivially satisfied.
	
	%\commentK{It is different then the \eqref{RTCond}, which it state holds for $ n\leq -1, \text{ and } k\leq 0$. The case $k=0$ is $g_{n,0}=  g_{n,0}$ which is satisfied, and if \eqref{RTCond} holds for $n \leq 0$, it holds for $n \leq -1$.}
	%So far, we proved (i) for $f \in C^2_0( \OM)$. 
	The proof for $\bbf \in L^1(\mathbf{S}^m; \OM) $ follows from the density of $C^2_0( \OM)$ in $L^1(\OM)$.
	
	\vspace{1cm}
	%%%%%%%%%%----------- SUFFICIENCY ----------

	(ii) {\bf Sufficiency of Theorem \ref {RangeCharac_Eventensor}:} Recall $m=2q$, $q\geq 0$. Given the double sequence $\{g_{n,k}\}$ for all odd $n \leq -2q-1$, and $k\in \BZ$, we construct a real valued symmetric $m$- tensor $\bbf$ in $\OM$ such that  the map on the torus
	$\ds \left\{
	\begin{array}{ll}
		X \bbf \text{ on }\Gam_+,\\
		-X \bbf \text{ on }\Gam_-\\
	\end{array}\right.$
	has the  Fourier coefficients matching the $\{g_{n,k}\}$'s.
	
	Recall the construction in \eqref{g_sequence_mEVEN},
	\begin{align}\label{gnk_bg1}
		g_{-n}(e^{i\beta}): = \sum_{k=-\infty}^\infty g_{-n,k} \, e^{i k \beta}, \; \emph{for odd } n \leq -2q-1,\quad e^{i  \beta}\in \Gam,
		\end{align}and define the sequence valued map on $\Gam$%$\bg^{\text{odd}}(e^{i  \beta})$
	\begin{align}\label{bg_gnk}
		\bg^{\text{odd}}(e^{i  \beta}) :=
			\langle g_{-(2q+1)}(e^{i  \beta}), g_{-(2q+3)}(e^{i  \beta}),  g_{-(2q+5)}(e^{i  \beta}),\cdots \cdot \rangle.
	\end{align}
	
	By the decay assumption \eqref{gnk_decay},  $\bg^{\text{odd}} \in l^{1,2}_{\INF}(\Gam) \cap C^{1,\mu}(\Gam; l^1)$; see 
	Lemma \ref{prop:bg_gnk} in the appendix. In particular, $\bg^{\text{odd}} \in  Y_{\mu}(\Gam)$ for $\mu>1/2$.
	
	%%%%%%%%%%%%%%%%%
	We use the Bukhgeim-Cauchy integral formula \eqref{BukhgeimCauchyFormula} to
 construct the sequence valued map $\bu^{\text{odd}} (z)$ inside $\OM$:
 	\begin{align}\label{u_construction}
		\bu ^{\text{odd}}  (z) = 	\langle u_{-(2q+1)}(z), u_{-(2q+3)}(z),  \cdots \rangle:= \B \left[ \bg^{\text{odd}}\right] (z), \quad z\in \OM.
	\end{align}

	By Theorem \ref{BukhgeimCauchyThm} (ii), the constructed $\bu ^{\text{odd}} \in C^{1,\mu}(\OM;l^{1}) \cap C^{\mu}(\ol \OM;l^{1})$  is $\LL$-analytic in $\OM$,
	%\begin{align}\label{bu_L2analytic1}
	%\dba\bu^{\text{odd}}(z) +L^2 \del\bu^{\text{odd}}(z) = \bzero,\quad z\in \OM.
	%\end{align}
	\begin{align}\label{vneg}
		\ol{\del} u_{n} + \del u_{n-2} = 0,\quad \text{for all odd } n\leq -2q-1.
	\end{align}

	While $\bu^{\text{odd}}$ constructed in \eqref{u_construction}  is $\LL$-analytic, in general, its trace $ \bu^{\text{odd}} \lvert_{\Gam}$  need not be equal to $\bg^{\text{odd}}$. 
	It is at this point that the constraints   \eqref{RTCond_Eventensor} come  into play.
	By using the hypothesis \eqref{RTCond_Eventensor},
	\begin{align*}%\label{RTCond_evenodd}
		g_{n,k}=  (-1)^k  g_{n+2k,-k}, \quad \text{for odd } \;   n \leq -2q-1,\;\text{and}\; k\leq -1,
		%\quad  \text{for all} \; n \geq 0, \;\text{and} \; k\leq -1. 
	\end{align*} and Theorem \ref{newmapping_Hilbert},	% \eqref{Hu_comp}, 
	we obtain 	$$\ds([I+i\HT] \bg^{\text{odd}})_{n,k}= 0, \text{ for all odd } n \leq -2q-1, \text{ and } k \in \BZ.$$

	Thus, $\ds[I+i\HT] \bg^{\text{odd}}= \bzero$, and the sufficiency part of Theorem \ref{NecSuf_BukhgeimHilbert_Thm} applies to yield% that the trace $\bu^{\text{odd}} \lvert_{\Gam}$ matches $\bg^{\text{odd}}$,
	\begin{align}\label{buodd-trace_bgodd}
		\bu^{\text{odd}} \lvert_{\Gam} = \bg^{\text{odd}}.
	\end{align}
	
	All of the positive Fourier modes $u_n, g_n$ for odd $n\geq 2q+1$  are constructed by conjugation,
	\begin{align}\label{construct_vpos}
		u_{n}&:=\ol{u_{-n}},\quad \text{in}\; \OM,  \\
		g_{n}&:=\ol{g_{-n}}, \quad \text{on}\; \Gam. 
	\end{align}
	Also, by conjugating \eqref{vneg} we note that the positive Fourier modes satisfy
	\begin{align*}%\label{vpos+}
		\ol{\del} u_{n+2} + \del u_{n} = 0,\quad \text{for all odd } n\geq 2q+1.
	\end{align*}Moreover, using \eqref{buodd-trace_bgodd} they extend continuously to $\Gam$ and
	\begin{align*}
		u_{n}|_\Gam=\ol{u_{-n}}|_\Gam=\ol{g_{-n}}=g_{n},\quad \text{odd } n\geq 2q+1.
	\end{align*}
	
	In summary, we have shown that 
	\begin{alignat}{2}\label{allmodes}
		&\ol{\del} u_{n} + \del{u_{n-2}} = 0,  && \quad \text{for all odd integers }  |n|\geq 2q+3,\\ \label{alluTrace_g}
		&u_{n}\lvert_{\Gam} = g_{n}, && \quad \text{for all odd integers }  |n|\geq 2q+1.
	\end{alignat}

In the case of the 0-tensor, $\bbf=f_0$, which is defined directly from $u_{-1}$ by $\ds f_0 =\del u_{-1}+\ol{\del u_{-1}}$.

We consider next the case $q\geq 1$ of tensors of order 2 or higher. 

Recall the construction \eqref{g_sequence_mEVEN},
	\begin{align}\label{gnk_bg11}
		g_{-2n+1}: = \sum_{k=-\infty}^\infty g_{-2n+1,k} \, e^{i k \beta}, \text{ for } 1\leq  n\leq q,
		\end{align}and define $g_{1}, g_{3}, ..., g_{2q-1}$ by conjugation
\begin{align}\label{g_conjOdd}
g_{2n-1}: = \ol{g_{-2n+1}}, 1\leq  n\leq q.	
\end{align}
Also recall	 the non-uniqueness class $\Psi_{g}^{\text{even}}$ in \eqref{NART_mEvenPsiClass}. 

For $\left( \psi_{-1}, \psi_{-3}, \cdots \psi_{-(2q-1)}\right)\in \Psi_{g}^{\text{even}}$ arbitrary, 
define the modes $u_{\pm 1}, u_{\pm 3}, ..., u_{\pm (2q-1)}$ in $\OM$ by
\begin{equation}\label{NART_mEven_uPsi}
	%\begin{aligned}
		u_{-(2n-1)} := \psi_{-(2n-1)} \text{ and } u_{2n-1} := \ol{\psi}_{-(2n-1)}, \quad 1 \leq n \leq q.
	%\end{aligned}
\end{equation}By the definition of the class \eqref{NART_mEvenPsiClass} and \eqref{g_conjOdd}, 
\begin{equation}\label{u1trace}
	\begin{aligned}
		u_{-(2n-1)} \lvert_{\Gam} &= g_{-(2n-1)}, \quad 1 \leq n \leq q, \; q \geq 1, \quad \text{and} \\
		u_{2n-1} \lvert_{\Gam} &= \ol{g}_{-(2n-1)}= g_{2n-1}, \quad 1 \leq n \leq q, \; q \geq 1.
	\end{aligned}
\end{equation}

The components of the $m$-tensor $\bbf$ are defined via the one-to-one correspondence between  $\{\tilde{f}_{2n}:\; -q\leq n\leq q\}$ in \eqref{eq:cov_tensor} and the functions  $\{f_{2n}:\; -q\leq n\leq q\}$ as follows. 

We define $f_{2q}$  by using $\psi_{-(2q-1)}$ from the non-uniqueness class, and $u_{-(2q+1)}$ 
from the Bukhgeim-Cauchy formula \eqref{u_construction},
\begin{align}\label{first}
2f_{2q}:= \ol{\del} \psi_{-(2q-1)}+ \del u_{-(2q+1)}.
\end{align}
Then, define $\{f_{2n}:\; 0 \leq n \leq q-1\}$ solely from the information in the non-uniqueness class via
\begin{align} 
2f_{2n}:= \ol{\del} \psi_{-(2n-1)} + \del \psi_{-(2n+1)},\;\;0 \leq n \leq q-1,\label{NART_mEven_fmPsiEq}
%2f_{0}&:= \ol{\del \psi_{-1}}+\del\psi_{-1}. \label{NART_mEven_fzeroPsiEq}
\end{align}Note that $f_0$ is real valued. 
Finally, define $\{f_{-2n}:\; 1\leq n\leq q\}$ by conjugation 
\begin{align}\label{lastDef_f}
f_{-2n}:=\ol{f_{2n}},\quad 1 \leq n \leq q.
\end{align}
By construction, $f_{2n}\in L^1(\OM)$, $-q\leq n\leq q$. Moreover, if $\psi_{-1},\cdots,\psi_{-2q+1}\in C^{1,\mu}(\OM)$, then  $f_{2n}\in C^\mu(\OM)$.

In the remaining of the proof  we check that, for the tensor $\bbf$ defined above, the Fourier coefficients of the map $\ds \left\{
	\begin{array}{ll}
		X \bbf \text{ on }\Gam_+,\\
		-X \bbf \text{ on }\Gam_-\\
	\end{array}\right.$ match the given $g_{n,k}$'s for all odd $n\in\BZ$ and $k\in \BZ$.

%The remaining steps require $q \geq 1$, and the construction of the tensor field $\bbf_{\psi} $  is in terms of the Fourier mode $u_{-(2q+1)}$ in \eqref{constructionODDSnegative} and the class $\Psi_{g}^{\text{even}}$ in \eqref{NART_mEvenPsiClass}.

	%Using $\bu^{\text{odd}} \in  C^{1,\mu}(\OM;l^{1}) \cap C^{\mu}(\ol \OM;l^{1})$  in \eqref{u_construction}, $\bg^{\text{odd}}$  in \eqref{bg_gnk}, and \eqref{construct_vpos}, we
	
	With $u_n$ defined in \eqref{u_construction} for odd $n\leq -2q-1$ (and by conjugation for odd $n\geq 2q+1$), and  with $\psi_{-2n+1}$ given by the non-uniqueness class (and $\psi_{2n-1}$ given by conjugation) for $0\leq n\leq q$, 
	we define the following two functions:
	\begin{align}\label{definitionU}
%		u^{\text{odd}}(z, e^{i\theta})&:=  \sum_{\substack{n\in \BZ\\ n  =\,\text{odd}}} u_{n}(z)e^{i n\tta}, \quad (z,\theta)\in \OM \times \sph, \\
		u^{\text{odd}}(z, e^{i\theta})&:= \sum_{|n| \geq q} u_{2n+1}e^{\i (2n+1)\tta} +
		 \sum_{n=1}^{q} \psi_{-(2n-1)} e^{ - \i (2n-1)\tta} + \sum_{n = 1}^{q} \ol{\psi}_{-(2n-1)}e^{\i (2n-1)\tta}
		%\sum_{\substack{n=1\\ n  =\,\text{odd}}}^{\INF}  u_{-n}(z)e^{-i n\tta} + \sum_{\substack{n=1\\ n = \,\text{odd}}}^{\INF}  u_{n}(z)e^{i n\tta},
		%\sum_{\substack{|n|\geq 2q+1}}  u_{n}(z)e^{i n\tta} +  \sum_{\substack{|n|< 2q+1}}  u_{n}(z)e^{i n\tta},
	\end{align}
	and
	\begin{align}\label{definitiong}
		g(z, e^{i\theta})&:=  \sum_{\substack{n=- \INF \\ n = \,\text{odd}}}^{\INF}  g_{n}(z)e^{i n\tta}, \quad (z,\theta)\in\Gam\times\sph.
		%\sum_{\substack{n=1\\ n = \,\text{odd}}}^{\INF}  g_{-n}(z)e^{-i n\tta} + \sum_{\substack{n=1\\ n = \,\text{odd}}}^{\INF}    g_{n}(z)e^{i n\tta},
	\end{align} 
%	%TAlso using $u_{-1}$ we define 
%	and 
%	\begin{equation}\label{fDefnTh3}
%		f(z) := \re \left ( \del u_{-1}(z) \right ),\quad z\in\OM.
%	\end{equation}
%	Note that $u_{-1}\in C^{1,\mu}(\OM)$, yields $f\in C^\mu(\OM)$.

	Since $\bu ^{\text{odd}} \in C^{1,\mu}(\OM;l^{1}) \cap C^{\mu}(\ol \OM;l^{1})$, we employ \cite[Corollary 4.1]{sadiqTamasan01} and
	\cite[Proposition 4.1 (iii)]{sadiqTamasan01} to infer the regularity $u^{\text{odd}} \in C^{1,\mu}(\OM \times \sph)\cap C^{\mu}(\ol{\OM}\times \sph)$. 
	
	In particular,  for each $e^{i\theta}\in \sph$, the trace of $u^{\text{odd}}(\cdot,e^{i\theta})$ on $\Gam$   satisfies
	%\begin{equation}\label{utrace_g}
	\begin{align} \label{utrace_g}
		u^{\text{odd}}(\cdot,e^{i\theta})\lvert_{\Gam} &=\left.\left ( \sum_{\substack{n= - \INF \\ n = \,\text{odd}}}^{\INF}    u_{n}e^{i n\tta}\right) \right\lvert_{\Gam}
		= \sum_{\substack{n= - \INF \\ n = \,\text{odd}}}^{\INF}   \left ( u_{n}\lvert_{\Gam} \right)  e^{i n\tta} 
		= \sum_{\substack{n=- \INF \\ n = \,\text{odd}}}^{\INF}    g_{n}e^{i n\tta} = g(\cdot,e^{i\theta}),
	\end{align} 
	%\end{equation}
	where the third equality above uses \eqref{alluTrace_g} and  \eqref{u1trace}. Note that, since $u^{\text{odd}}\in C^{\mu}(\ol{\OM}\times \sph)$, we can conclude now that $g$ defined in \eqref{definitiong} lies in $C^{\mu}(\Gam \times \sph)$.
	%Since $u^{\text{odd}} \in C^{\mu}(\ol{\OM}\times \sph)$, $g \in C^{\mu}(\Gam \times \sph)$.

%	Since $u\in C^{1,\mu}(\OM \times \sph)\cap C^{\mu}(\ol{\OM}\times \sph)$, 
	
%	the formal calculation in \eqref{eq:transport_fourier}  is now justified for each $z \in \OM$. 
%	For each direction $\btheta=\langle\cos\theta,\sin\theta\rangle$, we obtain

	Since the term by term differentiation in \eqref{definitionU} is now justified,
	
		\begin{align*}
		\btheta \cdot \nabla u^{\text{odd}} &=% \ol{\del} \; \ol{\psi_{-1}} +\del \psi_{-1} + 
		 \sum_{n=0}^{q-1} (\ol{\del} \psi_{-(2n-1)} + \del \psi_{-(2n+1)} )e^{ -\i (2n)\tta} + \sum_{n=1}^{q-1} (\ol{\del} \; \ol{\psi}_{-(2n+1)} +\del \ol{\psi}_{-(2n-1)})e^{ \i (2n)\tta} \\
		&\qquad + e^{-\i (2q) \tta} (\ol{\del}\psi_{-(2q-1)}+\del u_{-(2q+1)})   + e^{\i (2q) \tta} (\del \ol{\psi}_{-(2q-1)} + \ol{\del} \; \ol{u}_{-(2q+1)})\\
		% &\quad +\sum_{n = -\INF}^{\INF} ( \ol{\del}u_{2n}+\del u_{2n-2}) e^{i (2n-1) \tta}+\sum_{n = m}^{\INF} (\ol{\del}u_{-n}+ \del u_{-n-2})e^{-i (n+1) \tta} + \sum_{n = m}^{\INF} ( \ol{\del}u_{n+2}+\del u_{n}) e^{i (n+1) \tta}.
		&=  2 \sum_{k=-q}^{q} f_{2k}(z) e^{- \i (2k) \tta} =  2\langle \bbf,  \btheta^{2q} \rangle
	\end{align*}
	where the cancellation uses \eqref{allmodes}, and the second equality uses the definition of $f_{2k}$'s in  \eqref{first}, \eqref{NART_mEven_fmPsiEq}, and \eqref{lastDef_f}.

%	Since $u^{\text{odd}}\in C^{1,\mu}(\OM \times \sph)$,  the formal calculation in \eqref{eq:transport_fourier}  is now justified for each $z \in \OM$. 
%	For each direction $\btheta=\langle\cos\theta,\sin\theta\rangle$, we obtain
%	%Using $\tta\cdot\nabla v= e^{-i\fii} \ol{\del v} + e^{i\fii} (\del v)$, we obtain
%	\begin{align}\label{transport_uf}
%		\btheta\cdot\nabla u^{\text{odd}}(z, e^{i\theta})
%		&=
%		2 \re \left ( \del u_{-1}(z) \right ) + 2\re\left\{ \sum_{\substack{n=1\\ n = \,\text{odd}}}^{\INF}   \left(\dba u_{-n}(z) + {\del} u_{-n-2}(z)\right)e^{-in\tta}\right\} =2f(z),
%	\end{align}  
%	where the second equality uses \eqref{allmodes} and \eqref{fDefnTh3}.

	We remark next two properties of the function $g$ defined in \eqref{definitiong} which are needed later: Since only the odd modes  in the angular variable are used, $g(z,\cdot)$ is an odd function, and thus  the odd condition \eqref{sym**} is satisfied. In addition, we claim that $g$ also satisfies the symmetry condition \eqref{sym*}. Indeed,
	
	\begin{align*}g (e^{i (2 \theta -\beta -\pi)}, e^{i(\theta +\pi)}) &=\sum_{\substack{m=- \INF \\m = \,\text{odd}}}^{\INF}  \sum_{p \in \BZ} g_{m,p} e^{i m (\theta+\pi)}  e^{i p (2 \theta -\beta -\pi)}\\
		&=  \sum_{\substack{m=- \INF \\ m = \,\text{odd}}}^{\INF}   \sum_{p \in \BZ} (-1)^{m+p}g_{m,p}  e^{i (2p+m) \theta}  e^{-i p \beta }\\
		& \xlongequal{ m=n+2k, \;  p=-k}\sum_{\substack{n=- \INF \\ n = \,\text{odd}}}^{\INF}    \sum_{k \in \BZ}  (-1)^{n+k}g_{n+2k,-k} e^{i n \theta}  e^{i k \beta }\\
		& \xlongequal{ \eqref{RT_FourierOdd_Eventensor}} \sum_{\substack{n=- \INF \\ n = \,\text{odd}}}^{\INF}   \sum_{k \in \BZ} g_{n,k} e^{i n \theta}  e^{i k \beta }=g(e^{i
			\beta},e^{i\theta}).\numberthis\label{g_symm*}
	\end{align*}
	%where in the third equality we set $m+2p=n$ and  $-p=k$, while in the fourth equality we use the constraint \eqref{RT_FourierEven}.
	%We use \eqref{g_symm*} for $\theta$ replaced by $\theta+\pi$:
	
%	Since $\bu ^{\text{odd}} \in C^{1,\mu}(\OM;l^{1}) \cap C^{\mu}(\ol \OM;l^{1})$, by \cite[Corollary 4.1]{sadiqTamasan01} and
%	\cite[Proposition 4.1 (iii)]{sadiqTamasan01}, we conclude that $u^{\text{odd}} \in C^{1,\mu}(\OM \times \sph)\cap C^{\mu}(\ol{\OM}\times \sph)$. In particular,  for each $e^{i\theta}\in \sph$, the trace of $u^{\text{odd}}(\cdot,e^{i\theta})$  satisfies
%	%\begin{equation}\label{utrace_g}
%	\begin{align} \label{utrace_g}
%		u^{\text{odd}}(\cdot,e^{i\theta})\lvert_{\Gam} &=\left.\left ( \sum_{\substack{n= - \INF \\ n = \,\text{odd}}}^{\INF}    u_{n}e^{i n\tta}\right) \right\lvert_{\Gam}
%		= \sum_{\substack{n= - \INF \\ n = \,\text{odd}}}^{\INF}   \left ( u_{n}\lvert_{\Gam} \right)  e^{i n\tta} 
%		= \sum_{\substack{n=- \INF \\ n = \,\text{odd}}}^{\INF}    g_{n}e^{i n\tta} = g(\cdot,e^{i\theta}),
%	\end{align} 
%	%\end{equation}
%	where the third equality above uses \eqref{alluTrace_g}, \eqref{u1trace} and definition of $\psi_{-(2n-1)} \in \Psi^{even}_{g}$ for  $1 \leq n \leq q, \, q \geq 1$. Since $u^{\text{odd}} \in C^{\mu}(\ol{\OM}\times \sph)$, its trace $g \in C^{\mu}(\Gam \times \sph)$.
%	%Since $u^{\text{odd}} \in C^{\mu}(\ol{\OM}\times \sph)$, $g \in C^{\mu}(\Gam \times \sph)$.

	We have all the ingredients needed to show that $g$ coincides with $X\bbf$ on $\Gam_+$. In the following calculation we let 
	$ l(\beta, \theta)  = \lvert e^{i \beta} -e^{i (2\theta -\beta-\pi)}  \rvert$ denote the length of the chord in Figure \ref{fig:fanbeam1}, and use the geometric equality
	$e^{i\beta}-l(\beta,\theta)e^{i\theta}=e^{i(2\theta-\beta-\pi)}$. For each $(e^{i \beta},e^{i \theta}) \in \Gam_+$,
	\begin{align*}
		2 g(e^{i \beta},e^{i \theta}) &= g(e^{i \beta},e^{i \theta})- g(e^{i\beta},e^{i (\theta+\pi)})\\ 
		&= g(e^{i \beta},e^{i \theta})- g(e^{i (2\theta -\beta-\pi)},e^{i \theta})\\ 
		& =u^{\text{odd}}(e^{i \beta},e^{i \theta})- u^{\text{odd}}(e^{i (2\theta -\beta-\pi)},e^{i \theta}) 
		\\ 
		&
		= \int_{-  l(\beta, \theta) }^0  \btheta\cdot\nabla u ^{\text{odd}}(e^{i \beta}+t e^{i \theta}, e^{i\theta})dt\\ 
		&=\int_{- l(\beta, \theta) }^0 2 \langle \bbf ( e^{i \beta} +te^{i \theta}), \btheta^m \rangle dt
		= 2[X \bbf] (e^{i \beta},e^{i \theta}), \numberthis  \label{eq:g_Xf}
	\end{align*}
	where the first equality uses $g(e^{i\beta},\cdot)$ is angularly odd, the second equality uses the symmetry relation \eqref{g_symm*} with $\theta$ replaced by $\theta+\pi$, 
	the third equality uses \eqref{utrace_g},
	the fourth equality is the fundamental theorem of calculus, the fifth 
	equality uses \eqref{TransportEq_mEven}, and the last equality uses the support of $\bbf$ in $\OM$. 
	
	Therefore $g =X \bbf$ on $\Gam_+$, and, since $g$ is angularly odd, $g=-X \bbf$ on $\Gam_-$.
	
	The equation \eqref{eq:g_Xf} also shows that components of $m$- tensor field $\bbf$ integrates along all lines in the direction of $\btheta$, and thus $\bbf\in 	L^1(\mathbf{S}^m; \OM) $.
	
%	{\bf The construction of the tensor field $\bbf$ in the $q=0$ case.}
%	In the 0-tensor case ($q= 0$), there is no non-uniqueness class, and the method gives completes the unique determination of $0$-tensor by
%	\begin{align}\label{zerotensor_reconstruction}
%		&\bbf :=2 \re \del u_{-1},
%	\end{align}
%	where the odd Fourier mode $u_{-1}$ is as in \eqref{u_construction}.
	
	This finishes the proof of the even order tensor case.	 
	
	\qed  %End of the proof of the  Theorem in EVEN CASE

	The analytical reasoning in the proof of Theorem \ref{RangeCharac_Oddtensor} (the odd order tensor case) is the same as above. The change in parity of the order of the tensor merely modifies the algebraic statements. 
	We  sketch them below for the sake of completeness.

		(i) {\bf Necessity of Theorem \ref {RangeCharac_Oddtensor}:} Since $g$ is angularly even, \eqref{RTCond_even_Oddtensor} holds.
Since $g$ is real valued, \eqref{RTCond_reality_mOdd} holds. The identities \eqref{RT_FourierEven_Oddtensor} follow by direct calculation, see Lemma \ref{lem:evenness}.

	As in the even case, it suffices to prove \eqref{RTCond_Oddtensor} for  odd $m=(2q+1)$-tensor with components in $C^{2}_0(\OM)$. 
	%The result for components $f_{i_1i_2...i_m} \in L^1(\OM)$ follows by a density argument.
	%By using  the notations $\dba = (\del_{x_1}+i\del_{x_2})/2$, $\del =(\del_{x_1}-i\del_{x_2})/2$, and $\theta=\arg{\btheta}\in (-\pi,\pi]$,
	
	The transport equation \eqref{TransportEq1} has the right hand side modified as
	\begin{align}\label{TransportEq_mOdd}
	&[e^{-i\tta}\dba + e^{i\tta}\del] u(z,\btheta)=\sum_{n=0}^{q}\left( f_{2n+1} e^{-\i (2n+1)\tta} + f_{-(2n+1)} e^{\i (2n+1)\tta}\right), \quad (z,\btheta)\in \OM\times\sph.
	%\sum_{k=0}^{q} f_{2k} e^{-\i (2k)\tta} + \sum_{k=1}^{q} f_{-2k} e^{\i (2k)\tta}, 
	\end{align}
%Since $\bbf$ is real valued, $ \ol{f_{2n+1}} = f_{-(2n+1)}, \; -q \leq n\leq q, \,q \geq0$. 

	By Proposition \ref{prop:Uoddeven_mEvenOdd} (b),  the solution $u$ to \eqref{TransportEq_mOdd} %the boundary value problem \eqref{bvp_transport} 
	is an even function of $\btheta$, thus 
	%	$$u(z,\btheta) = - u(z, -\btheta).$$
	all of its Fourier coefficients (in the angular variable) of odd order vanish,
	$$u(z,\btheta)=\sum_{\substack{n \in \BZ\\	n  =\,\text{even}}}  u_{n}(z) e^{in\tta}.$$
	
	By identifying the Fourier modes of the same order, the equation \eqref{TransportEq_mEven} reduces to the system:
	\begin{align} \label{NART_mOdd_fmEq}
	&\ol{\del} u_{-2n}(z) + \del u_{-(2n+2)}(z) = f_{2n+1}(z), && 0 \leq n \leq q, \; q \geq 0, \\ \label{NARTmOddAnalyticEq_Even}
	&\ol{\del} u_{-2n}(z) + \del u_{-(2n+2)}(z) = 0, 
	&& n \geq q+1, \; q \geq 0,  
	%		\\ \label{NARTmOddAnalyticEq_Odd}
	%		&\ol{\del} u_{-(2n-1)}(z) + \del u_{-(2n+1)}(z) = 0, &&  n \geq 0.
	\end{align}
	
	Since $\bbf$ is real valued, the solution $\ds u(z,\btheta) $ of \eqref{TransportEq_mEven} is also  real valued, and its Fourier modes in the angular variable occur in conjugates, $u_{-n}=\ol{u_n}$. 
	Thus, it suffices to consider  the non-positive even Fourier coefficients of $u(z,\cdot)$: Let $\bu$ be  the sequence valued map
	\begin{align}\label{boldu2}
	\OM \ni z\mapsto  \bu(z)&: = \langle u_{-2}(z),u_{-4}(z),... \rangle.
	\end{align}
	
	Since all the components $f_{i_1i_2...i_m}  \in C^{2}_0(\OM)$, $X\bbf\in C^2(\Gam_-)$, and, thus, the solution $u$ to the transport problem \eqref{TransportEq_mOdd} is in $C^{2}(\ol \OM\times \sph)$. 
	Moreover, its trace $u \lvert_{\Gam \times \sph}\in C^2(\Gam\times\sph)$.

	For $z=e^{\i\beta}\in\Gam$, we use the negative even Fourier modes of the trace $g(e^{i\beta},\cdot)=u \lvert_{\Gam \times \sph}(e^{\i\beta},\cdot)$ to define the sequence valued map
	\begin{align}\label{eq:bg2}
		\bg(e^{i \beta}) := \langle g_{-2}(e^{i \beta}), g_{-4}(e^{i \beta}), ... \rangle, \mbox{ with } g_{-2n}(e^{i \beta}) = \frac{1}{2 \pi} \int_{-\pi}^{ \pi} g(e^{i \beta}e^{i\theta}) e^{i (2n)\theta} d\theta.
	\end{align}

	Since $u\in C^{2}(\ol \OM\times \sph)$, $\bu \in C^1(\ol \OM; l_1)$, and, thus,
	\begin{align}\label{trace2}
		\bg = \bu \lvert_{\Gam} \, \in C^1(\Gam; l_1) \subset  l^{1,1}_{\INF}(\Gam) \cap C^{\mu}(\Gam; l^1), \mu >1/2.\end{align}
	
	By \eqref{NARTmOddAnalyticEq_Even},	
	$\ds 	\LL^{q}\bu = 	\langle u_{-(2q+2)},u_{-(2q+4)},  \cdots \rangle$
	%\end{align*} 
	is $\LL$-analytic in $\OM$  and its trace $\LL^{q} \bg = \LL^{q}  \bu \lvert_{\Gam}\ds$ is the boundary value of an $\LL$-analytic map. 
	
	%Recall that $\HT$ is the Bukhgeim-Hilbert operator in \eqref{BHtransform}. 
	By the necessity part in Theorem \ref{NecSuf_BukhgeimHilbert_Thm}, we have $\HT (\LL^{q} \bg) \in C^{\mu}(\Gam; l_\INF)$ and
	%\begin{align}\label{Hg_char}
	$\ds	(I+i\HT)  (\LL^{q} \bg)= \bzero.$
	%\end{align}
	Since $\HT$ commutes with the left translation $\LL$, we obtained
	\begin{align}\label{Hg_char}
		\LL^{q} (I+i\HT)  (\bg)= \bzero.
	\end{align}
	In particular, for all even $n \leq -2q-2$, and $k \in \BZ$, we obtained
	\begin{align}\label{Hgnk_char2}
		([I+i\HT] \bg)_{n,k}= 0,
	\end{align}

	%	Since $\HT$ commutes with the left translation $\LL$, then 
	By Theorem \ref{newmapping_Hilbert}, the Fourier coefficients $\ds (\HT \bg)_{n,k}$, for even $ n \leq -2q$,  and  $k \in \BZ$,
	satisfy \eqref{Hu_comp}, 
	and thus
		\begin{align*}%\label{eq:HT_gnk}
		([I+i\HT] \bg)_{n,k}=\left\{
		\begin{array}{ll} 0 &\mbox{ if } k\geq 0,\\
			2g_{n,k}-2 (-1)^k  g _{n+2k,-k}&\mbox{ if } k\leq -1.
		\end{array}
		\right.\end{align*}In conjunction with \eqref{Hgnk_char2}, the Fourier coefficients of $\bg$ must satisfy \eqref{RTCond_Oddtensor}.
		
%		\begin{align}%\label{RTCond}
%				g_{n,k}=  (-1)^k  g_{n+2k,-k}, \text{ for all even } n\leq -2q,  \text{ and } k\leq -1.
%		\end{align}
%		Equation \eqref{RTCond_Oddtensor} for $k=0$ is trivially satisfied.
	
	%\commentK{It is different then the \eqref{RTCond}, which it state holds for $ n\leq -1, \text{ and } k\leq 0$. The case $k=0$ is $g_{n,0}=  g_{n,0}$ which is satisfied, and if \eqref{RTCond} holds for $n \leq 0$, it holds for $n \leq -1$.}
	%So far, we proved (i) for $f \in C^2_0( \OM)$. 
	The proof for $\bbf \in L^1(\mathbf{S}^m; \OM) $ follows from the density of $C^2_0( \OM)$ in $L^1(\OM)$.

	\vspace{1cm}
	%%%%%%%%%%----------- SUFFICIENCY ----------
	(ii) {\bf Sufficiency of Theorem \ref {RangeCharac_Oddtensor}:} Recall $m=2q+1$, $q\geq 0$. Given the double sequence $\{g_{n,k}\}$ for all even $n \leq -2q$, and $k\in \BZ$, we construct a real valued symmetric $m$- tensor $\bbf$ in $\OM$ such that  the map on the torus
	$\ds \left\{
	\begin{array}{ll}
		X \bbf \text{ on }\Gam_+,\\
		-X \bbf \text{ on }\Gam_-\\
	\end{array}\right.$
	has the  Fourier coefficients matching the $\{g_{n,k}\}$'s.
	
	Recall the construction in \eqref{g_sequence_mODD},
	\begin{align}\label{gnk_bg2}
		g_{-n}(e^{i\beta}): = \sum_{k=-\infty}^\infty g_{-n,k} \, e^{i k \beta}, \; \emph{for even } n \leq -2q,\quad e^{i  \beta}\in \Gam,
		\end{align}and define the sequence valued map on $\Gam$%$\bg^{\text{odd}}(e^{i  \beta})$
	\begin{align}\label{bg_gnk2}
		\bg^{\text{odd}}(e^{i  \beta}) :=
			\langle g_{-(2q)}(e^{i  \beta}), g_{-(2q+2)}(e^{i  \beta}),\cdots \cdot \rangle.
	\end{align}
	
	By the decay assumption \eqref{gnk_decay},  $\bg^{\text{even}} \in l^{1,2}_{\INF}(\Gam) \cap C^{1,\mu}(\Gam; l^1)$; see 
	Lemma \ref{prop:bg_gnk} in the appendix. In particular, $\bg^{\text{even}} \in  Y_{\mu}(\Gam)$ for $\mu>1/2$.

		We use the Bukhgeim-Cauchy integral formula \eqref{BukhgeimCauchyFormula} to
	construct the sequence valued map $\bu^{\text{even}} (z)$ inside $\OM$:
	\begin{align}\label{u_construction2}
		\bu ^{\text{even}}  (z) = 	\langle u_{-2q}(z), u_{-(2q+2)}(z),  \cdots \rangle:= \B \left[ \bg^{\text{even}}\right] (z), \quad z\in \OM.
	\end{align}

	By Theorem \ref{BukhgeimCauchyThm} (ii), the constructed $\bu ^{\text{even}} \in C^{1,\mu}(\OM;l^{1}) \cap C^{\mu}(\ol \OM;l^{1})$  is $\LL$-analytic in $\OM$,
	%\begin{align}\label{bu_L2analytic1}
	%\dba\bu^{\text{odd}}(z) +L^2 \del\bu^{\text{odd}}(z) = \bzero,\quad z\in \OM.
	%\end{align}
	\begin{align}\label{vneg2}
		\ol{\del} u_{n} + \del u_{n-2} = 0,\quad \text{for all even } n\leq -2q.
	\end{align}

	%While $\bu^{\text{even}}$ constructed in \eqref{u_construction2}  is $\LL$-analytic, in general, its trace $ \bu^{\text{even}} \lvert_{\Gam}$  need not be equal to $\bg^{\text{even}}$. 	It is at this point that the constraints   \eqref{RTCond_Oddtensor} come  into play.
	
	By using the hypothesis \eqref{RTCond_Oddtensor},
	%\begin{align*}%\label{RTCond_evenodd}
	%g_{n,k}=  (-1)^k  g_{n+2k,-k}, \quad \text{for even } \;   n \leq -2q,  \;\text{and}\; k\leq -1,
	%\quad  \text{for all} \; n \geq 0, \;\text{and} \; k\leq -1. 
	%\end{align*}
   and Theorem \ref{newmapping_Hilbert},	% \eqref{Hu_comp}, 
   we obtain 	%$$\ds([I+i\HT] \bg^{\text{even}})_{n,k}= 0, \text{ for all even } n \leq -2q, \text{ and } k \in \BZ.$$
   %Thus, 
   $\ds[I+i\HT] \bg^{\text{even}}= \bzero$. The sufficiency part of Theorem \ref{NecSuf_BukhgeimHilbert_Thm} now applies to yield% that the trace $\bu^{\text{odd}} \lvert_{\Gam}$ matches $\bg^{\text{odd}}$,
	\begin{align}\label{buodd-trace_bgeven}
	\bu^{\text{even}} \lvert_{\Gam} = \bg^{\text{even}}.
	\end{align}
	
	All of the positive Fourier modes $u_n, g_n$ for even $n\geq 2q, q\geq0,$  are constructed by conjugation,
	\begin{align}\label{construct_vpos2}
	u_{n}:=\ol{u_{-n}},\text{ in } \OM, \text{ and } g_{n}:=\ol{g_{-n}}, \text{ on } \Gam. 
	\end{align}
%	Also, by conjugating \eqref{vneg2} we note that the positive Fourier modes satisfy
%	\begin{align*}%\label{vpos+}
%	\ol{\del} u_{n+2} + \del u_{n} = 0,\quad \text{for all  even } n\geq 2q, q\geq 0.
%	\end{align*}Moreover, using \eqref{buodd-trace_bgeven} they extend continuously to $\Gam$ and
%	\begin{align*}
%	u_{n}|_\Gam=\ol{u_{-n}}|_\Gam=\ol{g_{-n}}=g_{n},\quad \text{even } n\geq 2q, q\geq 0.
%	\end{align*}
As in the even order tensor case, it is easy to check that	
	\begin{alignat}{2}\label{allmodes2}
	&\ol{\del} u_{n} + \del{u_{n-2}} = 0,  && \quad \text{for all even integers }  |n|> 2q, q \geq 0\\ \label{alluTrace_g2}
	&u_{n}\lvert_{\Gam} = g_{n}, && \quad \text{for all even integers }  |n|\geq 2q, q \geq 0.
	\end{alignat}
	
	Recall the construction \eqref{g_sequence_mODD},
	\begin{align}\label{gnk_bg22}
		g_{-2n}: = \sum_{k=-\infty}^\infty g_{-2n,k} \, e^{i k \beta}, \text{ for } 0\leq  n\leq q,
		\end{align}and define $g_{2}, g_{4}, ..., g_{2q}$ by conjugation
\begin{align}\label{g_conjEven}
g_{2n}: = \ol{g_{-2n}},\quad 1\leq  n\leq q.	
\end{align}
Also recall	 the non-uniqueness class $\Psi_{g}^{\text{odd}}$ in \eqref{NART_mOddPsiClass}. 

For $\left( \psi_{0}, \psi_{-2}, \cdots \psi_{-2q}\right)\in \Psi_{g}^{\text{odd}}$ arbitrary, 
define the modes $u_0, u_{\pm 2}, ..., u_{\pm 2q}$ in $\OM$ by
\begin{equation}\label{NART_mOdd_uPsi}
	%\begin{aligned}
		u_{-2n} := \psi_{-2n} \text{ and } u_{2n} := \ol{\psi}_{-2n}, \quad 0 \leq n \leq q.
		%\end{aligned}
\end{equation}By the definition of the class \eqref{NART_mOddPsiClass} and \eqref{g_conjEven}, 

\begin{equation}\label{u2trace}
		%\begin{aligned}
			u_{-2n} \lvert_{\Gam} = g_{-2n}, \quad -q \leq n \leq q.
			%u_{2n} \lvert_{\Gam} &= \ol{g}_{-2n}= g_{2n}, \quad 0 \leq n \leq q, \; q \geq 0.
		%\end{aligned}
	\end{equation}

The components of the $m$-tensor $\bbf$ are defined via the one-to-one correspondence between  $\{\tilde{f}_{\pm (2n+1)}:\;  0\leq n\leq q\}$ in \eqref{eq:cov_tensor} and the functions  $\{f_{\pm (2n+1)}:\; 0\leq n\leq q\}$ as follows.

Define the function $f_{2q+1}$  by using $\psi_{-2q}$ from the non-uniqueness class, and $u_{-(2q+2)}$ from the Bukhgeim-Cauchy formula \eqref{u_construction2}, via
	\begin{align}\label{first2}
	2f_{2q+1}:=\ol{\del} \psi_{-2q} + \del u_{-2(q+1)}.
	\end{align}
	Then define $\{f_{2n+1}:\; 0 \leq n \leq q-1\}$ solely from the information in the non-uniqueness class via
	\begin{align} \label{NART_mOdd_fmPsiEq}
		2f_{2n+1}:=\ol{\del} \psi_{-2n} + \del \psi_{-(2n+2)},\;\;0 \leq n \leq q-1.
	\end{align}
	Finally, define $\{f_{-2n-1}:\; 0\leq n\leq q\}$ by conjugation 
	\begin{align}\label{lastDef_f2}
		f_{-(2n+1)}:=\ol{f_{2n+1}}, \quad 0 \leq n \leq q.
	\end{align}
	By construction, $f_{\pm(2n+1)}\in L^1(\OM)$, $0\leq n\leq q$. Moreover, if $\psi_{0},\cdots,\psi_{-2q}\in C^{1,\mu}(\OM)$, then  $f_{\pm(2n+1)}\in C^\mu(\OM)$.

	In the remaining of the proof  we check that, for the tensor $\bbf$ defined above, the Fourier coefficients of the map $\ds \left\{
	\begin{array}{ll}
		X \bbf \text{ on }\Gam_+,\\
		-X \bbf \text{ on }\Gam_-\\
	\end{array}\right.$ match the given $g_{n,k}$'s for all even $n\in\BZ$ and $k\in \BZ$.
	
	As in the even order tensor case, 
	%With $u_n$ defined in \eqref{u_construction2} for even $n\leq -2q$ (and by conjugation for even $n\geq 2q$), and  with $\psi_{-2n}$ given by the non-uniqueness class (and $\psi_{2n}$ given by conjugation) for $0\leq n\leq q$,
	we define
	\begin{align}\label{definitionU2}
	%		u^{\text{odd}}(z, e^{i\theta})&:=  \sum_{\substack{n\in \BZ\\ n  =\,\text{odd}}} u_{n}(z)e^{i n\tta}, \quad (z,\theta)\in \OM \times \sph, \\
	u^{\text{even}}(z, e^{i\theta})&:= \sum_{|n| \geq q+1} u_{2n}(z)e^{\i (2n)\tta} +
	\sum_{n=0}^{q} \psi_{-2n}(z) e^{ - \i (2n)\tta} + \sum_{n = 0}^{q} \ol{\psi}_{-2n}(z)e^{\i (2n)\tta}
	%\sum_{\substack{n=1\\ n  =\,\text{odd}}}^{\INF}  u_{-n}(z)e^{-i n\tta} + \sum_{\substack{n=1\\ n = \,\text{odd}}}^{\INF}  u_{n}(z)e^{i n\tta},
	%\sum_{\substack{|n|\geq 2q+1}}  u_{n}(z)e^{i n\tta} +  \sum_{\substack{|n|< 2q+1}}  u_{n}(z)e^{i n\tta},
	\end{align}
	and 
	\begin{align}\label{definitiong2}
	g(z, e^{i\theta})&:=  \sum_{\substack{n=- \INF \\ n = \,\text{even}}}^{\INF}  g_{n}(z)e^{i n\tta}, \quad (z,\theta)\in\Gam\times\sph,
	%\sum_{\substack{n=1\\ n = \,\text{odd}}}^{\INF}  g_{-n}(z)e^{-i n\tta} + \sum_{\substack{n=1\\ n = \,\text{odd}}}^{\INF}    g_{n}(z)e^{i n\tta},
	\end{align}
	and the corresponding trace identity $\ds u^{\text{even}}(\cdot,e^{i\theta})\lvert_{\Gam}  =g(\cdot,e^{i\theta})$ holds.
%	\begin{align} \label{utrace_g2}
%	u^{\text{even}}(\cdot,e^{i\theta})\lvert_{\Gam} 
	%&=\left.\left ( \sum_{\substack{n= - \INF \\ n = \,\text{even}}}^{\INF}    u_{n}e^{i n\tta}\right) \right\lvert_{\Gam}
	%= \sum_{\substack{n= - \INF \\ n = \,\text{even}}}^{\INF}   \left ( u_{n}\lvert_{\Gam} \right)  e^{i n\tta} 
	%= \sum_{\substack{n=- \INF \\ n = \,\text{even}}}^{\INF}    g_{n}e^{i n\tta} 
%	= g(\cdot,e^{i\theta}).
%\end{align}	

	Since $\bu ^{\text{even}} \in C^{1,\mu}(\OM;l^{1}) \cap C^{\mu}(\ol \OM;l^{1})$, we employ \cite[Corollary 4.1]{sadiqTamasan01} and
\cite[Proposition 4.1 (iii)]{sadiqTamasan01} to infer the regularity $u^{\text{even}} \in C^{1,\mu}(\OM \times \sph)\cap C^{\mu}(\ol{\OM}\times \sph)$. 

%In particular,  for each $e^{i\theta}\in \sph$, the trace of $u^{\text{even}}(\cdot,e^{i\theta})$ on $\Gam$   satisfies
%	\begin{align} \label{utrace_g2}
%	u^{\text{even}}(\cdot,e^{i\theta})\lvert_{\Gam} &=\left.\left ( \sum_{\substack{n= - \INF \\ n = \,\text{even}}}^{\INF}    u_{n}e^{i n\tta}\right) \right\lvert_{\Gam}
%	= \sum_{\substack{n= - \INF \\ n = \,\text{even}}}^{\INF}   \left ( u_{n}\lvert_{\Gam} \right)  e^{i n\tta} 
%	= \sum_{\substack{n=- \INF \\ n = \,\text{even}}}^{\INF}    g_{n}e^{i n\tta} = g(\cdot,e^{i\theta}),
%\end{align} 
%\end{equation}
%where the third equality above uses \eqref{alluTrace_g2} and \eqref{u2trace}. Note that, since $u^{\text{even}} \in C^{\mu}(\ol{\OM}\times \sph)$, we can conclude now that $g$ defined in \eqref{definitiong2} lies in $C^{\mu}(\Gam \times \sph)$.

	%Since $u^{\text{even}}\in C^{1,\mu}(\OM \times \sph)$ and term by term differentiation is justified. 
	Term by term differentiation in  \eqref{definitionU2}, and an application of \eqref{allmodes2}, \eqref{NART_mOdd_uPsi},  \eqref{first2},   \eqref{NART_mOdd_fmPsiEq}, and \eqref{lastDef_f2} yields
%\begin{align*}
%	\btheta \cdot \nabla u &=  (\ol{\del} \psi_{-2q}+ \del u_{-2(q+1)})e^{\i (2q+1) \tta} + \sum_{n = 0}^{q-1} (\ol{\del} \psi_{-2n}+ \del \psi_{-(2n+2)})e^{\i (2n+1) \tta}\\
	% &\quad + \sum_{n=1}^{q-1} (\ol{\del} \psi_{1-2n} + \del \psi_{-1-2n} )e^{ -\i (2n)\tta} + \sum_{n=1}^{q-1} (\ol{\del} \; \ol{\psi_{-1-2n}} +\del \ol{\psi_{1-2n}})e^{ \i (2n)\tta} \\
	% &\quad + e^{-\i (2q) \tta} (\ol{\del}\psi_{-(2q-1)}+\del u_{-(2q+1)})   + e^{\i (2q) \tta} (\del \ol{\psi_{-(2q-1)}} + \ol{\del} \; \ol{u_{-(2q+1)}}) \\
%	&\quad + (\del \ol{\psi}_{-2q}+ \ol{\del u}_{-2(q+1)})e^{-\i (2q+1) \tta} +\sum_{n = 0}^{q-1} (\del \ol{\psi}_{-2n}+ \ol{\del \psi}_{-(2n+2)})e^{-\i (2n+1) \tta}.
	% \\&\quad + \sum_{|n| \geq  2q+1} ( \ol{\del}u_{-n}+\del u_{-(n+2)}) e^{ - \i (n+1) \tta}.
%\end{align*}
%Now using the definition of 
%$f_{2k+1}$ for $-q \leq k \leq q$ in   \eqref{NART_mOdd_fmPsiEq}  in the calculation above, we conclude that $u$ defined in \eqref{definitionU2}
%solves the transport equation \eqref{TransportEq_mEven}
\begin{align*}
		\btheta \cdot \nabla u^{\text{even}} =  2\langle \bbf,  \btheta^{2q+1} \rangle.
\end{align*}

	%We remark next two properties of the function $g$ defined in \eqref{definitiong2} which are needed later: 
	Since only even modes are used in the angular variable, $g(z,\cdot)$ is an even function, and thus  the even condition \eqref{skew**} is satisfied. In addition, we claim that $g$ also satisfies the symmetry condition \eqref{skew*}. Indeed,
	\begin{align*}g (e^{i (2 \theta -\beta -\pi)}, e^{i(\theta +\pi)}) &=\sum_{\substack{m=- \INF \\m = \,\text{even}}}^{\INF}  \sum_{p \in \BZ} g_{m,p} e^{i m (\theta+\pi)}  e^{i p (2 \theta -\beta -\pi)}\\
	&=  \sum_{\substack{m=- \INF \\ m = \,\text{even}}}^{\INF}   \sum_{p \in \BZ} (-1)^{m+p}g_{m,p}  e^{i (2p+m) \theta}  e^{-i p \beta }\\
	& \xlongequal{ m=n+2k, \;  p=-k}\sum_{\substack{n=- \INF \\ n = \,\text{even}}}^{\INF}    \sum_{k \in \BZ}  (-1)^{n+k}g_{n+2k,-k} e^{i n \theta}  e^{i k \beta }\\
	& \xlongequal{ \eqref{RT_FourierEven_Oddtensor}} - \sum_{\substack{n=- \INF \\ n = \,\text{even}}}^{\INF}   \sum_{k \in \BZ} g_{n,k} e^{i n \theta}  e^{i k \beta }= - g(e^{i
		\beta},e^{i\theta}).\numberthis\label{g_symm*2}
	\end{align*}
	%where in the third equality we set $m+2p=n$ and  $-p=k$, while in the fourth equality we use the constraint \eqref{RT_FourierEven}.
	%We use \eqref{g_symm*} for $\theta$ replaced by $\theta+\pi$:

	%We have all the ingredients needed to show that $g$ coincides with $X\bbf$ on $\Gam_+$. In the following calculation we let 
%	$ l(\beta, \theta)  = \lvert e^{i \beta} -e^{i (2\theta -\beta-\pi)}  \rvert$ denote the length of the chord in Figure \ref{fig:fanbeam1}, and use the geometric equality
%	$e^{i\beta}-l(\beta,\theta)e^{i\theta}=e^{i(2\theta-\beta-\pi)}$. 
	
	For each $(e^{i \beta},e^{i \theta}) \in \Gam_+$,
	\begin{align*}
	2 g(e^{i \beta},e^{i \theta}) &= g(e^{i \beta},e^{i \theta})+g(e^{i\beta},e^{i (\theta+\pi)})\\ 
	&= g(e^{i \beta},e^{i \theta})- g(e^{i (2\theta -\beta-\pi)},e^{i \theta})\\ 
	& =u^{\text{even}}(e^{i \beta},e^{i \theta})- u^{\text{even}}(e^{i (2\theta -\beta-\pi)},e^{i \theta}) 
	\\ 
	&
	= \int_{-  l(\beta, \theta) }^0  \btheta\cdot\nabla u ^{\text{even}}(e^{i \beta}+t e^{i \theta}, e^{i\theta})dt\\ 
	&=\int_{- l(\beta, \theta) }^0 2 \langle \bbf ( e^{i \beta} +te^{i \theta}), \btheta^m \rangle dt
	= 2[X \bbf] (e^{i \beta},e^{i \theta}), \numberthis  \label{eq:g_Xf_mOdd}
	\end{align*}
	where the first equality uses $g(e^{i\beta},\cdot)$ is angularly even, the second equality uses the symmetry relation \eqref{g_symm*2} with $\theta$ replaced by $\theta+\pi$, 
	%the third equality uses \eqref{utrace_g2},
	%the fourth equality is the fundamental theorem of calculus, 
	the fifth equality uses \eqref{TransportEq_mOdd}, and the last equality uses the support of $\bbf$ in $\OM$. 
	
	Therefore $g =X \bbf$ on $\Gam_+$, and, since $g$ is angularly even, $g=-X \bbf$ on $\Gam_-$.
	
	The equation \eqref{eq:g_Xf_mOdd} also shows that components of the constructed $m$- tensor field $\bbf$ integrates along all lines in the direction of $\btheta$, and thus $\bbf\in 	L^1(\mathbf{S}^m; \OM) $.
	
	\qed  %End of the proof of the  Theorem in ODD CASE

	%%%%%%%%%%%%%%%%%%%%%%%%%%%%%%%%%%%%%%%%%%%%%%%%%%
	%%%%%%%%%%%%%%%%%%%%%%%%%%%%%%%%%%%%%%%%%%%%%%%%%%
	%%%%%%%%%%%%%%%%%%%%%%%%%%%%%%%%%%%%%%%%%%%%%%%%%%

	\section*{Acknowledgment}
	The work of K.~ Sadiq  was supported by the Austrian Science Fund (FWF), Project P31053-N32, and by the FWF Project F6801–N36 within the Special Research Program SFB F68 “Tomography Across the Scales”. 
	The work of A.~Tamasan  was supported in part by the National Science Foundation DMS-1907097.
	%%%%%%%%%%%%%%%%%%%%%%%%%%%%%%%%%%%%%%%%%%%%%%%%%%%%%%%%%%%%
	
	\appendix
	
	\section{Elementary results}
	To improve the readability, we moved the proof of the more elementary claims to this section. The presentation follows the order of their occurrence.
		
	\begin{lemma}\label{Qlemma}
		Let $m,k\geq 0$ be integers with $0 \leq k\leq m$, and let $Q_{m.k} (t) =  \left( t+ 1 \right)^{m-k} \left( t - 1 \right)^k$.
		Then,
		$\ds \{Q_{m,k}(t)\}_{k=0}^{m}$ form an basis in the space of polynomials of degree $m$.
	\end{lemma}
	\begin{proof}
		Given $h_0,h_1, \cdots ,h_m$, we show that there are unique $g_0,g_1, \cdots, g_m$ such that 
		\begin{align}\label{changePolyn}
			\sum_{k=0}^m g_k (t-1)^k (t+1)^{m-k} = \sum_{k=0}^m h_k t^k.\end{align}
		We argue by induction in $m$. For $m=1$, $ g_0 = (h_1+h_0)/2$ and  $g_1 = (h_1-h_0)/2$. 
		
		Assume next that $\ds \{Q_{m,k}(t)\}_{k=0}^{m-1}$ form a basis for polynomials of degree $m-1$. To simplify notation, let $a_k:= {m \choose k  } =\frac{m!}{k!(m-k)!}$, for $0\leq k\leq m$. Note that 
		$\ds \sum_{k=0}^{m} a_k=\sum_{k=0}^{m} {m \choose k  }  = 2^m$.

		%	Given $(h_0,h_1, \cdots ,h_m)$,  need to find $(g_0,g_1, \cdots, g_m)$ such that $\ds \sum_{k=0}^m g_k (t-1)^k (t+1)^{m-k} = \sum_{k=0}^m h_k t^k$.
		The left hand side of \eqref{changePolyn} rewrites
		\begin{align*}
			\sum_{k=0}^m &g_k (t-1)^k (t+1)^{m-k} = (t+1)\sum_{k=0}^{m-1} g_k (t-1)^k (t+1)^{m-1-k} +  g_m (t-1)^m  \\
			&\xlongequal{ \text{induction hypothesis} }(t+1)\sum_{k=0}^{m-1} \gamma_k t^k +    \sum_{k=0}^m (-1)^{m-k} a_kg_m t^k  \\
			&= \sum_{k=0}^{m-1} \gamma_k t^{k+1} + \sum_{k=0}^{m-1} \gamma_k t^k +   \sum_{k=0}^m (-1)^{m-k} a_k g_m  t^k  \\
			&= \left(\gamma_0 + (-1)^{m} a_0 g_m\right)+  \sum_{k=1}^{m-1} \left(\gamma_{k-1} + \gamma_k + (-1)^{m-k} a_k g_m \right)  t^k +  \left(\gamma_{m-1} + a_m g_m \right)  t^m.
			%& = \sum_{k=0}^m h_k t^k
		\end{align*}

		The identity \eqref{changePolyn} yields that $ \gamma_0, \gamma_1, \cdots, \gamma_{m-1}, g_m$ solve the $(m+1)\times (m+1)$ linear system:
		\begin{align}\label{linearsystem}
			\begin{bmatrix}
				1 & 0 & 0&0  &\cdots& 0 & 0 & (-1)^{m} a_0\\
				1 & 1 & 0&0 &\cdots & 0& 0 & (-1)^{m-1} a_1\\
				0 & 1 & 1&0	 &\cdots & 0& 0 & (-1)^{m-2} a_2\\
				0 & 0 & 1&1 &\cdots & 0& 0 & (-1)^{m-3} a_3\\
				\vdots & \vdots & \vdots & \ddots& \ddots & \vdots & \vdots & \vdots\\
				0 & 0 & 0&0	 &\cdots &1 &1 &- a_{m-1}\\
				0 & 0 & 0&0	 &\cdots &0 &1 & a_{m}\\
			\end{bmatrix}
			\begin{bmatrix}
				\gamma_0 \\ \gamma_1\\ 	\gamma_2\\\gamma_3\\ \vdots\\ \gamma_{m-1} \\ g_m 
			\end{bmatrix}
			=	
			\begin{bmatrix}
				h_0 \\ h_1\\ 	h_2\\ h_3\\ \vdots\\ h_{m-1} \\ h_m 
			\end{bmatrix}.
		\end{align}
			
		The determinant of the matrix above is calculated by expanding it along the last column: Since all the $m\times m$- cofactor matrices have determinant 1, the $(m+1)\times (m+1)$ determinant is $\ds\sum_{k=0}^m	a_k=2^m>0$. 
		
		In addition to $g_m$ being determined, the determination of $\gamma_0,...,\gamma_{m-1}$ together with the induction hypothesis uniquely determine $g_0, ... g_{m-1}$. In fact this argument can be construed in a recursive computation as follows. 
		Since $g_m=2^{-m}\sum_{k=0}^m(-1)^kh_k$, the unknowns $\gamma_{m-i}$ are to be determined recursively  for $i=1,...,m$, and thus the problem is reduced to the $m\times m$ case.

	\end{proof}
	
	Recall  $\OM = \{ z  \in \BC : |z| <1  \}$ is the complex unit disc,  $\Gam = \{ z  \in \BC : |z| =1  \}$ is its boundary, and $\sph$ is the set of unit directions.
	
	\begin{prop}\label{prop:fLp_regularity}
		Let $\bbf \in L^1(S^m;\OM)$ be a symmetric $m$-tensor with integrable components. Assume that each of its components satisfies one of the conditions:
 		\begin{align}\label{eq:f_refularity_cond}
			\textnormal{supp }f_{i_1...i_m}  \subset \{ z: \lvert z \rvert \leq\sqrt{ 1- \delta^2}  \}, \;  0 < \delta <1, \quad \textnormal{or} \quad  f_{i_1...i_m}\in L^p(\OM), \;p >2,
		\end{align}
		then $X\bbf  \in L^1(\Gamma\times\sph)$.
	\end{prop}
	
	%%%%%%%%%%%%
	\begin{proof}To fix ideas, we consider the even tensor case, when  $m=2l$. The odd tensor case follows similarly.
	
	From the  one-to-one linear combination correspondence between the components $f_{i_1...i_m}$ of $\bbf$, and the functions $f_{2k}$'s,  $-l\leq k\leq l$ in the identity \eqref{eq:innerprod_fThetaM},
 	\begin{equation*}%\label{eq:innerprodAPPENDIX}
		\langle \bbf, \btheta^m \rangle   = \sum_{k=-l}^{l} f_{2k} e^{-\i (2k)\tta} ,	
	\end{equation*}
	we have that each of the $f_{2k}$'s also satisfies $f_{2k}\in L^1(\OM)$ and
	\begin{align}\label{eq:f_refularity_condBIS}
			\textnormal{supp }f_{2k} \subset \{ z: \lvert z \rvert \leq\sqrt{ 1- \delta^2}  \}, \;  0 < \delta <1, \quad \textnormal{or} \quad  f_{2k}\in L^p(\OM), \;p >2.
		\end{align}

		For $e^{i \beta} \in \Gam$ and $e^{i\theta} \in \sph$,  the $X$-ray transform of $\bbf$ (with components extended by 0 outside $\OM$) is given by
		\begin{align*}
			X\bbf(e^{i \beta}, e^{i \theta}) = \sum_{k=-l}^l\int_{-\INF}^{\INF} f_{2k}(e^{i \beta} +t e^{i \theta})e^{-i2k\theta} dt.%= \int_{-\INF}^{\INF} f^{\theta}_{2k}(e^{i (\beta-\theta)} +t )e^{-i2k\theta} dt,
		\end{align*}
		
		We will show that each term in the sum above is integrable on the torus. To simplify notation, we drop the index notation from the function and show that if $f$ satisfies \eqref{eq:f_refularity_condBIS}, then
		%\Gam\times\sph\ni(e^{i \beta}, e^{i \theta})\mapsto 
		$$X_k f(e^{i \beta}, e^{i \theta}):=\int_{-\INF}^{\INF} f(e^{i \beta} +t e^{i \theta})e^{-i2k\theta} dt \mbox{ lies in }L^1(\Gam\times\sph).$$ 
		Rewrite $\ds X_k f(e^{i \beta}, e^{i \theta}):=\int_{-\INF}^{\INF} f_\theta(e^{i (\beta-\theta)} +t)e^{-i2k\theta} dt$, where $f_\theta (z) := f(ze^{i \theta})$ is obtained from $f$ by a rotation of the domain by angle $\theta$. Note that $f_\theta$ preserves the  $L^p$-norm of $f$ for any $p\geq1$.
		
		We estimate 
		\begin{align*}
			\lnorm{X_kf}&_{L^1(\Gamma\times\sph)} = \frac{1}{(2\pi)^2}    \int_{-\pi}^{\pi}
			\int_{-\pi}^{\pi} \lvert  X_kf(e^{i \beta}, e^{i \theta}) \rvert d \beta d \theta \\
			& \leq  \frac{1}{(2\pi)^2}  \int_{-\pi}^{\pi}
			\int_{-\pi}^{\pi} \int_{-\INF}^{\INF} \lvert  f_{\theta}(e^{i (\beta-\theta)} +t ) \rvert  dt  d \beta d \theta  \\
			%&   \xlongequal{ \alpha = \beta-\theta }  \frac{1}{(2\pi)^2}   \int_{-\pi}^{\pi} \int_{-\pi}^{\pi} \int_{-\INF}^{\INF} \lvert  f_{\theta}(e^{i \alpha} +t) \rvert  dt d \alpha d \theta \\
			% &  = \frac{1}{(2\pi)^2}    \int_{-\pi}^{\pi}
			% \int_{-\pi/2}^{\pi/2} \int_{-\INF}^{\INF} \lvert  f_{\theta}(e^{i \alpha} +t) \rvert  dt d \alpha d \theta 
			% + \frac{1}{(2\pi)^2}    \int_{-\pi}^{\pi}
			% \int_{\lvert \alpha \rvert > \pi/2} \int_{-\INF}^{\INF} \lvert  f_{\theta}(e^{i \alpha} +t) \rvert  dt d \alpha d \theta \\
			&  \xlongequal{ \alpha = \beta-\theta } 2\frac{1}{(2\pi)^2}    \int_{-\pi}^{\pi}
			\int_{-\pi/2}^{\pi/2} \int_{-\INF}^{\INF} \lvert  f_{\theta}(e^{i \alpha} +t ) \rvert  dt d \alpha d \theta \\
			&  \xlongequal{s = \sin \alpha } \frac{1}{2\pi^2}    \int_{-\pi}^{\pi}
			\int_{-\INF}^{\INF} \int_{-1}^{1}  \frac{\lvert f_{\theta}(\sqrt{1-s^2}+t + is ) \rvert }{\sqrt{1-s^2}}  ds dt d \theta \\
			&  \xlongequal{u = t + \sqrt{1-s^2} }\frac{1}{2\pi^2}    \int_{-\pi}^{\pi}
			\int_{-\INF}^{\INF} \int_{-1}^{1} \frac{ \lvert  f_{\theta}(u +is ) \rvert  }{\sqrt{1-s^2}}    ds du d \theta \\  \numberthis \label{eq:Xfnorm}
			&	= \frac{1}{2\pi^2}    \int_{-\pi}^{\pi}
			\int_{-1}^{1} \int_{-1}^{1} \frac{ \lvert  f_{\theta}(u+i s ) \rvert  }{\sqrt{1-s^2}}   ds du d \theta,
		\end{align*}	where the last equality uses $ \supp f_\theta \subset \OM$ for any $\theta\in(-\pi,\pi]$.
		
		On the one hand, if $\supp f  \subset  \{ z: \lvert z \rvert \leq\sqrt{ 1- \delta^2}  \}$,  then 
		\begin{align*}
			\lnorm{X_kf}_{L^1(\Gamma\times\sph)} & \leq \frac{1}{2\pi^2}    \int_{-\pi}^{\pi}
			\int_{-1}^{1} \int_{-\sqrt{1-\delta^2}}^{\sqrt{1-\delta^2}} \frac{ \lvert  f_{\theta}(u+is ) \rvert  }{\sqrt{1-s^2}}    ds du d \theta \\
			%& \leq \frac{1}{2\pi^2}    \frac{1}{\delta} \int_{-\pi}^{\pi}
			%\int_{-\INF}^{\INF}\int_{-\INF}^{\INF} \lvert  f_{\theta}(u, s ) \rvert     ds du d \theta \\
			& \leq \frac{1}{2\pi^2}    \frac{1}{\delta} \int_{-\pi}^{\pi}
			\lnorm{f_{\theta}}_{L^1(\OM)} d \theta 
			%= \frac{1}{2\pi^2}    \frac{1}{\delta} \int_{-\pi}^{\pi} \lnorm{f}_{L^1(\OM)} d \theta 
			= \frac{1}{\pi \delta }  \lnorm{f}_{L^1(\OM)}.
		\end{align*}	
		On the other hand, if $f \in L^p(\OM)$, $p >2$, let $T:=(-1,1)\times(-1,1)$ denote the unit square. 
		Since $\OM\subset T$, for every $\theta\in(-\pi,\pi]$, $\ds f_\theta\in L^p(T)$, and  $\ds \lVert f_\theta\lVert_{L^p(T)}=\lVert f\lVert_{L^p(\OM)}$. 
		
		Let $q=\frac{p}{p-1}$ be the conjugate index of $p$. If $p>2$, then $q<2$ and the map 
		$$ T  \ni(u,s) \mapsto \frac{ 1 }{\sqrt{1-s^2}}\mbox{ lies in } L^q(T).$$ 
		Moreover, since it is constant in one coordinate, $\ds \lnorm{\frac{ 1}{\sqrt{1-(\cdot)^2}}}_{L^q(T)}=2\lnorm{\frac{ 1}{\sqrt{1-(\cdot)^2}}}_{L^q( -1,1)}$. 
		
		An application of the Hölder's inequality in \eqref{eq:Xfnorm} yields
		\begin{align*}
			\lnorm{X_kf}_{L^1(\Gamma\times\sph)} & \leq \frac{1}{\pi^2} \int_{-\pi}^{\pi}
			\lnorm{\frac{ 1}{\sqrt{1-(\cdot)^2}}}_{L^q( -1,1)} \lnorm{f_{\theta}}_{L^p(\OM)} d \theta 
			%& = \frac{1}{2\pi^2} \int_{-\pi}^{\pi}
			%\lnorm{f}_{L^p(\OM)} \lnorm{\frac{ 1}{\sqrt{1-s^2}}}_{L^q( -1,1)} d \theta 
			= \frac{2}{\pi} \lnorm{\frac{ 1}{\sqrt{1-(\cdot)^2}}}_{L^q(T)}
			\lnorm{f}_{L^p(\OM)} .
		\end{align*}

	\end{proof}
	
	%%%%%%%%%%

	%%%%%%%%%%
	\begin{lemma}\label{lem:evenness}
		Let $m \geq 0$ be an integer.
		Let $g$ be an integrable function on the torus satisfying the symmetry relation
		\begin{align}\label{sym*1}
			g (e^{i \beta},e^{i \theta})=(-1)^m g (e^{i(2\theta -\beta-\pi)},e^{i( \theta+\pi)}),\text{ for } (e^{i \beta},e^{i \theta})\in\Gam\times\sph,
		\end{align}
		and $g_{n,k}$'s 
		%		$$g_{n,k} = \frac{1}{(2 \pi)^2} \int_{-\pi}^{\pi} \int_{-\pi}^{\pi} g(e^{i \beta}, e^{i\theta}) e^{-i n \theta}  e^{-i k \beta} d \theta d \beta,\quad n,k \in \BZ$$
		be its Fourier coefficients. Then 
		\begin{align}\label{FourierEvenness_cond}
			g_{n,k}= (-1)^m (-1)^{n+k}g_{n+2k,-k}, \mbox{ for all }n,k\in\BZ.
		\end{align} 
	\end{lemma}
	
	\begin{proof}
		
		%	The symmetry in \eqref{sym*}, and the  change of variables $(\beta,\theta)\leftrightarrow (2\theta -\beta-\pi,\theta+\pi)$ in \eqref{eq:gnk} yields
		%The Fourier coefficients $	g_{n,k}$ for $n,k \in \BZ,$ of $g$ are given by 
		Indeed,
		\begin{align*}%\label{eq:gnk1}
			g_{n,k} &= \frac{1}{(2 \pi)^2} \int_{-\pi}^{\pi} \int_{-\pi}^{\pi} g(e^{i \beta}, e^{i\theta}) e^{-i n \theta}  e^{-i k \beta} d \theta d \beta \\
			&\xlongequal{  \eqref{sym*1} } (-1)^m 
			\frac{1}{(2 \pi)^2} \int_{-\pi}^{\pi} \int_{-\pi}^{\pi} g (e^{i(2\theta -\beta-\pi)},e^{i( \theta+\pi)}) e^{-i n \theta}  e^{-i k \beta} d \theta d \beta \\
			& \xlongequal{ \gamma = \theta +\pi } (-1)^m (-1)^{ n}
			\frac{1}{(2 \pi)^2} \int_{-\pi}^{\pi} \int_{0}^{2\pi} g (e^{i(2\gamma -\beta-3\pi)},e^{i \gamma})  e^{-i n \gamma}  e^{-i k \beta} d \gamma d \beta \\
			&\xlongequal{ \text{periodicity} } (-1)^m  (-1)^{ n}
			\frac{1}{(2 \pi)^2} \int_{-\pi}^{\pi} \int_{-\pi}^{\pi} g (e^{i(2\gamma -\beta-\pi)},e^{i \gamma})  e^{-i n \gamma}  e^{-i k \beta} d \gamma d \beta \\
			%&=  (-1)^{ n}
			%\frac{1}{(2 \pi)^2} \int_{-\pi}^{\pi} \int_{-\pi}^{\pi} g (e^{i(2\gamma -\beta-\pi)},e^{i \gamma})    e^{-i k \beta} e^{-i n \gamma} d \beta d \gamma \\
			&\xlongequal{ \alpha = 2\gamma -\beta-\pi} (-1)^m   (-1)^{ n}
			\frac{1}{(2 \pi)^2} \int_{-\pi}^{\pi} \int_{2\gamma}^{2\gamma-2\pi} g (e^{i \alpha},e^{i \gamma})  e^{-i k (2\gamma-\alpha-\pi)} e^{-i n \gamma} (-d \alpha) d \gamma \\
			&= (-1)^m  (-1)^{ n+k}	 \frac{1}{(2 \pi)^2} \int_{-\pi}^{\pi} \int_{2\gamma-2\pi}^{2\gamma} g (e^{i \alpha},e^{i \gamma})   e^{-i (n+2k)  \gamma} e^{i k \alpha} d \alpha d \gamma \\
			%	 &=  (-1)^{ n+k}
			%	 \frac{1}{(2 \pi)^2} \int_{-\pi}^{\pi}  e^{-i (n+2k)  \gamma}  \left\{ \int_{2\gamma-3\pi}^{2\gamma-\pi} g (e^{i \alpha},e^{i \gamma})  e^{i k \alpha}  d \alpha \right\} d \gamma \\
			%	  &\xlongequal{ \text{using periodicity} } 
			%	   (-1)^{ n+k}
			%	  \frac{1}{(2 \pi)^2} \int_{-\pi}^{\pi}  e^{-i (n+2k)  \gamma}  \left\{ \int_{-\pi}^{\pi} g (e^{i \alpha},e^{i \gamma})  e^{i k \alpha}  d \alpha \right\} d \gamma \\
			&\xlongequal{ \text{periodicity} } 
			(-1)^m 	(-1)^{ n+k}
			\frac{1}{(2 \pi)^2} \int_{-\pi}^{\pi}   \int_{-\pi}^{\pi} g (e^{i \alpha},e^{i \gamma})    e^{-i (n+2k)  \gamma} e^{i k \alpha} d \alpha  d \gamma \\
			&=(-1)^m (-1)^{n+k}g_{n+2k,-k}.
		\end{align*} 
	\end{proof}

	The following result is a direct consequence of  the uniform convergence. Recall that 
	$\sph$ denotes the unit circle.
	%%%%%%%%%%%%
	\begin{lemma}\label{lem_nfn} 
		Let $(X,\lnorm{\cdot})$ be a Banach space, $\{a_n\}_{n \in \BZ}$ be a sequence in $X$, and $p\geq 0$ integer.
		%, and $p$ be a positive integer.
		%If  $\{\jpn^p \widehat{f}(n) \}_{n \in \BZ} \in l^1(\BZ;X)$, i.e.
		If $\ds \sum_{n \in \BZ} \jpn^p \lnorm{ a_n} < \INF$, then $ \ds \sph \ni \zeta \mapsto  \sum_{n \in \BZ} a_n \zeta^{ n}$ defines a $C^p$ map on $\sph$  with values in $X$.
	\end{lemma}
	%%%%%%%%%%%

	For the following result, we recall some of the spaces in \eqref{spaces}, for $0<\mu<1$, $p=1,2$:
	\begin{equation}\label{spaces1}
		\begin{aligned} 
			l^{1,p}_{\INF}(\Gam) &:= \left \{ \bg:=\langle  
			g_{0}, g_{-1} , g_{-2}, \cdots  \rangle\; : \lnorm{\bg}_{l^{1,p}_{\INF}(\Gam)}:= \sup_{\xi \in \Gam}\sum_{j=0}^{\INF}  \jpj^p \lvert g_{-j}(\xi) \rvert < \INF \right \},\\
			C^{\mu}(\Gam; l_1) &:= \left \{ \bg:=\langle  
			g_{0}, g_{-1} , g_{-2}, \cdots  \rangle:
			\sup_{\xi\in \Gam} \lVert \bg(\xi)\rVert_{\ds l_{1}} + \underset{{\substack{
						\xi,\eta \in \Gam \\
						\xi\neq \eta } }}{\sup}
			\frac{\lVert \bg(\xi) - \bg(\eta)\rVert_{\ds l_{1}}}{|\xi - \eta|^{ \mu}} < \INF \right \}, 
			%\\
			%Y_{\mu} &:= \left \{ \bg: \bg \in  l^{1,2}_{\INF}(\Gam) \; \text{and} \;
			%\underset{{\substack{
			%			\xi,\eta \in \Gam \\
			%			\xi\neq \eta } }}{\sup} \sum_{j=0}^{\INF}  \jpj 
			%\frac{\lvert g_{-j}(\xi) - g_{-j}(\eta)\rvert }{|\xi - \eta|^{ \mu}} < \INF \right \},
		\end{aligned}
	\end{equation} where, for brevity, we use the notation $\jpj=(1+|j|^2)^{1/2}$.
	%Similarly,  we define $ C^{\mu}(\ol \OM; l_\INF) $ and $ \ds C^{\mu}(\OM; l_\INF) = \bigcup_{0<r<1} C^{\mu}(\ol \OM_r; l_\INF)$, where for $0<r<1$, $\OM_r = \{ z \in \BC : |z| <r  \}$.
	% The following result recalls the necessary and sufficient conditions for a sufficiently regular map to be the boundary value of an $L^2$-analytic function.
	
	\begin{lemma}\label{prop:bg_gnk}  
		%	Let $0\leq \mu \leq1$, and 
		Let $\{g_{-n,k}\}_{n \geq 0,  k\in \BZ}$ be a double sequence  satisfying  the decay
		\begin{align}\label{gnk_decay_1}
			\sum_{n=0}^{\INF}  \jpn^{2} \sum_{k=-\INF}^{\INF} \lvert g_{-n,k} \rvert < \INF, \quad \text{and} \quad
			% \sum_{n=0, \text{odd}}^{\INF}  \jpn^2 \sum_{k=-\INF}^{\INF} \lvert g_{-n,k} \rvert < \INF, \quad \text{and} \quad
			\sum_{k=-\INF}^{\INF}  \jpk^{1+\mu} 
			\sum_{n=0}^{\INF} \lvert g_{-n,k} \rvert < \INF,  \; \text{for some} \; 0\leq \mu \leq1.
		\end{align}
		Let  $\ds  \bg( \zeta) :=\langle  
		g_{0}( \zeta ), g_{-1}(\zeta ), g_{-2}(\zeta ), \cdots \cdot \rangle,$ for $\zeta \in \Gam=\{z: |z|=1\}$, and where for each $n \geq 0$,  
		$\ds g_{-n}(\zeta):=\sum_{k=-\infty}^\infty g_{-n,k} \, \zeta^{ k }.$
		Then $\bg \in l^{1,2}_{\INF}(\Gam) \cap C^{1,\mu}(\Gam; l^1)$.
	\end{lemma}
	\begin{proof}
		The fact that $\bg \in l^{1,2}_{\INF}(\Gam)$ follows from the first bound in \eqref{gnk_decay_1}:
		%	Using  \eqref{gnk_decay_1}, the norm
		\begin{align*}
			\lnorm{\bg}_{l^{1,2}_{\INF}(\Gam)} = \sup_{\zeta  \in \Gam}
			\sum_{n=0}^{\INF}   \jpn^2 \lvert g_{-n}( \zeta ) \rvert 
			\leq  \sum_{n=0}^{\INF} \jpn^2  \sum_{k=-\INF}^{\INF} \left \lvert g_{-n,k} \right \rvert
			< \INF.
		\end{align*}
		We prove that $\bg \in C^{1,\mu}(\Gam; l^1)$ by interpolation between  $C^{1}(\Gam; l^1)$, and  $C^{1,1}(\Gam; l^1)$. 
		For each $k \in \BZ$, consider the $l^1$ sequence  $\bF(k) := \{g_{-n,k}\}_{n \geq0}$.
		By the second bound in \eqref{gnk_decay_1}, we have 
		\begin{align*}
			\sum_{k\in \BZ}  \jpk^{1+\mu} \lnorm{\bF(k)}_{l^1} = \sum_{k=-\INF}^{\INF}   \jpk^{1+\mu} \sum_{n=0}^{\INF} \left \lvert g_{-n,k} \right \rvert
			< \INF.
		\end{align*}
		%	i.e. $ \ds \BZ \ni k  \mapsto  \jpk^{1+\mu}\bF(k) $ belongs to $l^1(\BZ;l^1(\Gam))$.

		Since $\{\jpk^{1+\mu} \bF(k) \}_{k \in \BZ}  \in l^1(\BZ_k;l^1(\BN_n))$, by applying Lemma \ref{lem_nfn} for $X=l^1(\BN_n)$, and $p = 1+\mu $, with $\mu =0$, we conclude  $\bg \in C^{1}(\Gam; l^1)$, whereas with $\mu =1$, we obtain  $\bg \in C^{1,1}(\Gam; l^1)$.
		The result for $0 < \mu <1$ follows by interpolation.
		
	\end{proof}
	%%%%%%%

	%%%%%%%%%%%%%%%%%%%%

	%%%%%%%%%%%%%%%%%%%%%%%%%%%%%%%%%%%%%%%%%%%%%%%%%%%%%%%%%%%%%%%%%%%%%%%%%%
\end{document}